\documentclass[journal,onecolumn,12pt]{IEEEtran}
\usepackage{amsthm}
\usepackage{amsmath}
\usepackage{amssymb}
\usepackage{bbm}
\usepackage{url}
\usepackage{amscd}
\usepackage{mathrsfs}
\usepackage{graphicx}
\usepackage{xcolor}
\usepackage{float}
\usepackage[margin=0.7in]{geometry}

\usepackage[all]{xy}

\newtheorem{theorem}{Theorem}

\newtheorem{lemma}[theorem]{Lemma}
\newtheorem{proposition}[theorem]{Proposition}

\newtheorem{definition}{Definition}
\usepackage{setspace}
\newcommand{\ep}{\varepsilon}
\newcommand{\te}{\theta}

\newcommand{\F}{\mathbb{F}}

\renewcommand{\P}{\mathbb{P}}
\newcommand{\supp}{\mathrm{supp}}
\renewcommand{\l}{\ell}

\renewcommand{\a}{\underline{a}}
\renewcommand{\c}{\underline{c}}

\title{Classification of a class of planar quadrinomials}
\author{Chin Hei Chan\thanks{C. Chan is at the Dept. of Mathematics, Hong Kong University of Science and Technology, Clear Water Bay, Kowloon, Hong Kong (email: chchanam@connect.ust.hk).} and Maosheng Xiong\thanks{M. Xiong is at the Dept. of Mathematics, Hong Kong University of Science and Technology, Clear Water Bay, Kowloon, Hong Kong (email: mamsxiong@ust.hk). }}
\date{}

\begin{document}

\maketitle
\begin{abstract}
Let $p$ be an odd prime, $k,\l$ be positive integers, $q=p^k, Q=p^{\l}$. Following main ideas from \cite{Ding4}, we characterise planar functions of the form $f_{\c}(X)=c_0X^{qQ+q}+c_1X^{qQ+1}+c_2X^{Q+q}+c_3X^{Q+1}$ over $\F_{q^2}$ for any $\c=(c_0,c_1,c_2,c_3) \in \F_{q^2}^4$ in terms of linear equivalence. 

\end{abstract}
\section{Introduction}
\subsection{Background and motivation}

Let $p$ be a prime number, $k$ a positive integer, $q=p^k$, and $\F_{q}$ the finite field of order $q$. A function $F: \F_q \to \F_q$ is called \emph {planar} if all the equations
\begin{eqnarray} \label{1:pn} F(x+a)-F(x)=b, \qquad \forall a,b \in \F_q, a \ne 0\end{eqnarray}
have exactly one solution. In other words, for any $a \in \F_q^*:=\F_q \setminus \{0\}$, the function
\[D_a F(x)=F(x+a)-F(x),\]
called the \emph{derivative of $F$ in the direction of $a$}, is a permutation on $\F_q$.

In the above definition, only the addition operation is involved, so planar functions can also be defined on any finite dimensional vector space over $\F_p$.

Planar functions were originally introduced by Dembowski and Ostrom in the seminar paper \cite{Dem} in connection with projective planes in finite geometry. This notion coincides with that of perfect nonlinear (PN) functions in odd characteristic introduced by Nyberg \cite{Nyb} from cryptography. From this perspective, planar functions have the best differential uniformity, hence if used as S-boxes, they offer the best resistance to differential cryptanalysis, one of the most powerful attacks known today used against block ciphers. Besides their importance in cryptography, planar functions have applications in coding theory \cite{Car-2,YuanJ}, combinatorics \cite{Kanat,WengG1} and some engineering areas \cite{DingC1}. Planar functions also correspond to important algebraic and combinatorial structures such as commutative semifields \cite{Coulter,ZP}. All of these make the study of planar functions fruitful and important in a much broader context in mathematics and computer science. Planar functions together with almost perfect nonlinear (APN) functions in characteristic two have become a central topic in design theory, coding theory and cryptography (\cite{Blo,Car-B,DingC,Pot}).

Planar functions exist only for odd $q$. Currently there are less than 20 distinct infinite families of planar functions (see \cite{Hau} for a list and \cite{Chen1} for a recent construction of planar functions). One of the main reasons why new planar functions are so difficult to construct and analyze is that planar functions are classified up to a certain notion of equivalence, namely linear equivalence, EA-equivalence or CCZ-equivalence, and to show that a given planar function is equivalent or inequivalent to some known ones is usually rather difficult, most of such verification involves quite technical computation. For a flavor of the techniques, interested readers may refer to a recent work \cite{Shi}.

Now let $p$ be an odd prime, $k,\l$ be some positive integers, $q=p^k, Q=p^{\l}$, and $\F_{q^2}$ the finite field of order $q^2$. For any $\c:=(c_0,c_1,c_2,c_3) \in \F_{q^2}^4$, we define a quadrinomial $f_{\c}(X) \in \F_{q^2}[X]$ given by
\begin{equation}\label{DO}
f_{\c}(X)=c_0X^{qQ+q}+c_1X^{qQ+1}+c_2X^{Q+q}+c_3X^{Q+1}.
\end{equation}
In this paper we study planar functions from these $f_{\c}(X)$ for all $\c \in \F_{q^2}^4$. 

We remark that when $q$ is even, this class of quadrinomials $f_{\c}(X)$ and some variations have been studied extensively in the literature. In fact in this case these $f_{\c}(X)$'s have been the main object of study in more than 40 papers and by more than 60 authors, culminating in a complete understanding as to when $f_{\c}(X)$ is a permutation on $\F_{q^2}$ (see \cite{Ding4}, and the result is actually on functions slightly more general than $f_{\c}(X)$ ). In another perspective, $f_{\c}(X)$ satisfies the subfield property
$$f_{\c}(aX)=a^{Q+1}f_{\c}(X) \text{ for all }a \in \F_q.$$
By identifying $(x,y) \in \F_q^2$ with $X=x+\zeta y \in \F_{q^2}$ for a fixed element $\zeta \in \F_{q^2} \setminus \F_q$, one sees that the class of $f_{\c}(X)$ is linear equivalent to the class of $(Q,Q)$-biprojective functions $f^*(x,y) \in \F_q[x,y]^2$ given by
\begin{equation}\label{Biprojective}
f^*(x,y)=\left(a_0x^{Q+1}+a_1x^Q y+a_2xy^Q+a_3y^{Q+1},b_0x^{Q+1}+b_1x^Q y+b_2xy^Q+b_3y^{Q+1}\right).
\end{equation}
In this language, G\"{o}lo\u{g}lu also provided a complete characterization of permutation from $f_{\c}(X)$ under linear equivalence \cite{Faruk}, and a complete characterization of APN from $f_{\c}(X)$ under linear equivalence \cite{Faruk2}, hence giving a satisfactory answer to a question of Carlet \cite{Car-0}. We shall remark that many other properties of $f_{\c}(X)$ such as the Boomerang uniformity \cite{Li,Wu} can be easily derived from \cite{Ding4,Faruk}.

In view of all these papers and in particular of \cite{Car-0} and \cite{Faruk2}, it is natural for us to study and to characterize planar functions from $f_{\c}(X)$ for odd $q$.

\subsection{Statement of main results}

The main result of the paper is as follows.


\begin{theorem}\label{Main}
    Let $p$ be an odd prime, $k$ and $\l$ positive integers, $q=p^k$ and $Q=p^\l$. Consider the quadrinomial $f_{\c}(X)$ given in (\ref{DO}) for any $\c=(c_0,c_1,c_2,c_3) \in \F_{q^2}^4$. Then $f_{\c}(X)$ is planar over $\F_{q^2}$ if and only if it is linear equivalent to one of the polynomials listed below (which are either univariate over $\F_{q^2}$ or $(Q,Q)$-biprojective over $\F_q^2$):
    \begin{enumerate}
        \item $X^{Q+1}$, where $\frac{\l}{\gcd(k,\l)}$ is even;
        \item $X^{Q+q}$, where $\frac{k\l}{\gcd(k,\l)^2}$ is odd;
        \item $P_2(x,y)=(x^Q y, x^{Q+1}+\ep y^{Q+1})$, where $\frac{k}{\gcd(k,\l)}$ is odd and $\ep \in \F_q^*$ is a non-square.
        \end{enumerate}
    Moreover, if $k \mid \l$, then $f_{\c}(X)$ is planar if and only if it is linear equivalent to $X^2$.
\end{theorem}

We first remark that in Theorem \ref{Main}, since $f_{\c}(X)$ is a Dembowski-Ostrom (DO for short) polynomial, when $f_{\c}(X)$ is planar, the CCZ-equivalence, EA-equivalence and linear equivalence all coincide with each other \cite{Lilya1}. Since planar functions in odd characteristic are natural analogue of APN functions in characteristic two, Theorem \ref{Main} can be considered as both complementing and parallel to \cite[Theorem 1.1]{Faruk2}, which gave a complete classification of APN functions from the class of $f_{\c}(X)$ for even $q$ (by using the language of $(Q,Q)$-biprojective functions). Moreover, we remark that planar functions from Families 1)--3) are all known: $X^{Q+1}$ and $X^{Q+q}$ resemble the Albert family \cite{Albert}, and $P_2(x,y)$ is a subclass of the Zhou-Pott family \cite{ZP}. This means that no new planar function can be found in the family (\ref{DO}). This is similar to the case of even characteristic where it was shown in \cite{Faruk2} that no new APN can be found in this family as well.

Next, we explain the method we use in proving Theorem \ref{Main}. The polynomial $f_{\c}(X)$ can be written as $f_{\c}(X)=X^{Q+1} A\left(X^{q-1}\right)$ where \[A(X)=c_0X^{Q+1}+c_1X^{Q}+c_2X+c_3.\]
Define
\[B(X)=c_3^qX^{Q+1}+c_2^qX^{Q}+c_1^qX+c_0^q, \quad g(X)=B(X)/A(X). \]
It was well-known that permutation properties of $f_{\c}(X)$ are closely related to those of the accompanying rational function $g(X)$ defined over $\mu_{q+1}$, the set of $(q+1)$-th roots of unity in $\F_{q^2}$. Encompassing this idea, many techniques were developed to study this $g(X)$ over $\mu_{q+1}$ in the literature, most of the techniques were elementary but quite complex and involved a lot of computation. In a recent paper \cite{Ding4}, Ding and Zieve provided a new way of studying $g(X)$: they employed advanced tools such as the Hurwitz genus formula from arithmetic geometry to study geometric properties of $g(X)$ (i.e. the type of branch points and ramification indices) from which permutation properties of $f_{\c}(X)$ follow in some natural way. This powerful technique allowed them to resolve eight conjectures and open problems from the literature concerning $f_{\c}(X)$ for $q$ even and to cover most of the previous results. In proving Theorem \ref{Main}, we adopt their ideas to study $f_{\c}(X)$ for odd $q$. This is the main ingredient of the paper. While the geometric properties of $g(X)$ for a general $q$ including $q$ odd were essentially present in \cite{Ding4}, to prove Theorem \ref{Main}, however, we still need some additional information. In particular, we need a complete and detailed classification of $g(X)$ in terms of geometric conditions which we obtain by following the ideas of \cite{Ding4} closely, from which we obtain seven linear equivalence classes of $f_{\c}(X)$ under general conditions (see Theorem \ref{AE2}). We shall comment that the idea of linear equivalence was also hinted in the proof of \cite[Theorem 1.2]{Ding4}. The next step is to study these seven linear equivalence classes of polynomials and check which ones are planar functions. It turns out that there are only three planar functions up to linear equivalence as shown in Theorem \ref{Main}. While most of the cases are straightforward to check, the last two cases are much more complicated. We rely on the Weil bound of character sums over finite fields to rule them out as planar functions. 

We remark that in some parts of the paper we follow the presentation from \cite{Ding4}. 

The paper is organized as follows. In Section \ref{Pre} we introduce some universal notations and recall some background results that are needed for our proofs. In Section \ref{Cla} we give detailed geometric properties of the rational function $g(X)$ under consideration. Most results in this section were essentially contained in the proofs of \cite[Theorem 3.1]{Ding4}, but they were not written down explicitly, and it seems not easy to pin down exactly where they are in \cite{Ding4}. For the sake of readers, we write down the exact statement which will be required in the proof of Theorem \ref{Main} and produce a detailed proof following the ideas of \cite{Ding4}, with a focus on the case that $q$ is odd. Then in Section \ref{Cla-2} we classify $g(X)$ in terms of linear equivalence. In Section \ref{AEf} we derive a list of seven linear equivalence classes of $f_{\c}(X)$ under the classification of $g(X)$. Then in Section \ref{ProMain} we prove Theorem \ref{Main} by checking which functions among the seven linear equivalence representatives of $f_{\c}(X)$ are planar. We treat the first 5 easy cases in this section and leave the more complicated two representatives in Section \ref{Appen} {\bf Appendix}. In Section \ref{ProMain} we also study the special case when $k \mid \l$, which is also straightforward after all this preparation. Finally in Section \ref{Con} we conclude our paper.

\section*{Acknowledgment}
The main ideas of this paper mainly follow from \cite{Ding4}. The authors would like to thank Zhiguo Ding and Michael Zieve for bringing their attention to \cite{Ding4} in the first place which was not published at the time. More importantly, the authors would like to thank Ding and Zieve for generous sharing and in particular for their careful and illuminating explanation of some technical details of \cite{Ding4} before and during writing of this paper. 

The authors also want to express their gratitude to the anonymous referees who provided critical information regarding a previous version of the paper. Without their valuable feedback the paper in its current form would not be possible. C. Chan would like to thank the Department of Mathematics at Hong Kong University of Science and Technology for financial support. The research of M. Xiong was supported by RGC grant number 16307524 from Hong Kong.

\section{Preliminaries}\label{Pre}
Throughout this paper, we adopt the following notation:
\begin{itemize}
\item for any finite set $S$, $\# S$ is the cardinality of $S$;

\item for any field $K$, $K^*=K \setminus \{0\}$, $\P^1(K)= K \cup \{\infty\}$ is set of $K$-rational points in $\P^1$, and $\overline{K}$ is an algebraic closure of $K$;

\item $p$ is an (odd) prime, $k$ is a positive integer, $q=p^k$, $\F_q$ is the finite field of order $q$;

\item for any positive integer $d$, $\mu_d$ is the set of $d$-th roots of unity in $\overline{\F}_q$; 

\item for any $c \in \F_{q^2}$, $\bar{c}=c^q$; 


\end{itemize}

\subsection{Self-conjugate reciprocal polynomials}
Let $D(X) \in \F_{q^2}[X]$. Denote by $D^{(q)}(X)$ the polynomial in $\F_{q^2}[X]$ formed by taking $q$-th powers (or conjugates) on all the coefficients of $D(X)$. The \emph{conjugate reciprocal} of $D(X)$ is defined as
\[\widehat{D}(X):=X^{\deg D}D^{(q)}(1/X).\]

To be more precise, if $D(X)=\sum_{i=0}^r a_iX^i$ with $a_i \in \F_{q^2}$ for all $i$ and $a_r \neq 0$ where $r>0$, then $D^{(q)}(X)=\sum_{i=0}^r \bar{a}_iX^i$ and $\widehat{D}(X)=\sum_{i=0}^r \bar{a}_{r-i}X^i$. Note that if $D(0)=a_0=0$, then $\deg \widehat{D} < r=\deg D$. Otherwise if $D(0) \ne 0$ then we have $\deg \widehat{D}=\deg D$.


A nonzero polynomial $D(X) \in \F_{q^2}[X]$ is called \emph{self-conjugate reciprocal} (SCR for short) if $\widehat{D}(X)=\alpha D(X)$ for some $\alpha \in \F_{q^2}^*$. In particular this implies $D(0) \neq 0$.

The following about conjugate reciprocals and SCR polynomials are immediate from the above definitions:

\begin{lemma}\label{SCR}
All of the following hold:
\begin{itemize}

\item if $D(X) \in \F_{q^2}[X]$ is \emph{SCR} then $\widehat{D}(X)/D(X) \in \mu_{q+1}$;

\item if $D(X) \in \F_{q^2}[X]$ is nonzero and $\alpha \in \overline{\F}_q^*$, then the multiplicity of $\alpha$ as a root of $D(X)$ equals the multiplicity of $\alpha^{-q}$ as a root of $\widehat{D}(X)$;

\item $D(X) \in \F_{q^2}[X]$ is \emph{SCR} if and only if the multiset of roots of $D(X)$ is preserved by the function $\alpha \mapsto \alpha^{-q}$. In particular, if $\deg D=1$ then it is \emph{SCR} if and only if its unique root is in $\mu_{q+1}$;

\item if $\alpha \in \F_{q^2}^*$ and $\beta \in \F_q$ then $\alpha X^2+\beta X+\bar{\alpha}$ is \emph{SCR}.
\end{itemize}
\end{lemma}
There are simple conditions describing the nature of the roots of a degree-2 SCR polynomial, but the situation is quite different for $p=2$ and for $p \ge 3$ being odd. Here we focus on the case $p \ge 3$. Interested readers may refer to \cite[Lemma 2.4]{Ding4} for the case $p=2$.
\begin{lemma}\label{lem1}
    Assume $p$ is odd and $D(X)=\alpha X^2+\beta X+\bar{\alpha}$ with $\alpha \in \F_{q^2}$ and $\beta \in \F_q$ not both zero. Define $\Delta(D):=\beta^2-4\alpha\bar{\alpha}$. Then the following hold:
    \begin{enumerate}
        \item $D(X)$ has a multiple root (which must be in $\mu_{q+1}$) if and only if $\Delta(D) = 0$;
        \item $D(X)$ has two distinct roots in $\mu_{q+1}$ if and only if $\Delta(D)$ is a non-square in $\F_q^*$;
        \item $D(X)$ has no roots in $\mu_{q+1}$ if and only if $\Delta(D)$ is a square in $\F_q^*$.
    \end{enumerate}
\end{lemma}
This result is quite elementary, since we cannot find it in the literature, for the sake of completeness, we provide a proof here.
\begin{proof}
If $\alpha=0$ and $\beta \neq 0$, then $D(X)=\beta X$ is a degree-one polynomial, with a unique root 0, which is not in $\mu_{q+1}$. We also have $\Delta(D)=\beta^2$ being a square in $\F_q^*$. So 3) applies to this case.

Hence from now on we assume $\alpha \neq 0$. By Lemma \ref{SCR}, $D(X)$ is a degree-2 SCR polynomial. It is clear that $D$ has a multiple root if and only if $\Delta(D)=0$, and by Lemma \ref{SCR}, since its multiset of roots is preserved by the function $\gamma \mapsto \gamma^{-q}$, this multiple root must satisfy $\gamma=\gamma^{-q}$, that is, $\gamma \in \mu_{q+1}$.

Now suppose $\Delta(D) \neq 0$. Then $D(X)$ has two distinct roots $\gamma,\delta$. Note that $\gamma\delta=\frac{\bar{\alpha}}{\alpha} \in \mu_{q+1}$, so either both $\gamma,\delta$ are in $\mu_{q+1}$, or none of them are in $\mu_{q+1}$. Noting that $\Delta(D) \in \F_q^*$, there is an $\theta \in \F_{q^2}^*$ such that $\theta^2=\Delta(D)$. Easy to see that we have either $\bar{\te}=\te$ or $-\te$, according to whether $\te \in \F_q^*$ or not. We may take $\gamma=-\frac{\beta+\te}{2\alpha}$. We see that $\gamma \in \mu_{q+1}
$ if and only if $\gamma^{q+1}=1$, that is,
$$\left(\frac{\beta+\te}{2\alpha}\right)\left(\frac{\beta+\bar{\te}}{2\bar{\alpha}}\right)=1.$$
This can be further simplified as
$$\beta^2+\beta(\te+\bar{\te})+\te\bar{\te}=4\alpha\bar{\alpha},$$
and again
$$\beta(\te+\bar{\te})+\te\bar{\te}+\te^2=0.$$
Taking $\bar{\te}=\ep\te$ where $\ep \in \{ \pm 1\}$ we have
$$\te(1+\ep)(\beta+\te)=0.$$
Since $\te \neq 0$ and $\beta+\te \neq 0$ as $\gamma \neq 0$, this implies that $\ep=-1$, that is $\theta \notin \F_{q}^*$. Hence we conclude that in this case $\gamma, \delta \in \mu_{q+1}$ if and only if $\te \notin \F_q^*$. This completes the proof of Lemma \ref{lem1}.
\end{proof}

\subsection{Rational Functions}
Let $K$ be a field and $G(X)=N(X)/D(X)$ be a rational function in $K$ where $N,D \in K[X]$ and $D$ is nonzero. Let $C(X)=\gcd(N(X),D(X))$, the monic greatest common divisor of $N(X)$ and $D(X)$ in $K[X]$. We write $N(X)=C(X)N_0(X), D(X)=C(X)D_0(X)$ with $N_0, D_0 \in K[X]$, so that $\gcd(N_0(X),D_0(X))=1$. We identify $G(X)$ with $G_0(X)=N_0(X)/D_0(X)$ and view $G(X)$ as the function $\P^1(K) \to \P^1(K)$ defined by $\alpha \mapsto G_0(\alpha)$, so $G(X)$ is also well-defined at elements $\alpha \in K$ even if $N(\alpha)=D(\alpha)=0$. We refer to $N_0$ and $D_0$ as the numerator and denominator of $G$ respectively, and define the degree of $G$ as $\deg G=\max\{\deg N_0,\deg D_0\}$ if $G(X) \neq 0$. $G$ is called \emph{separable} if the field extension $K(x)/K(G(x))$ is a separable extension of function fields where $x$ is transcendental over $K$. In fact, $G$ is separable if and only if $G'(X) \neq 0$ if and only if $G \notin K(X^p)$ where $p=\mathrm{char}(K)$ \cite[Lemma 2.2]{Ding5}.

We say non-constant $F,G \in K(X)$ are \emph{linear equivalent over $K$} (or \emph{$K$-linear equivalent} for short) if there are degree-one $\rho,\sigma \in K(X)$ such that $G=\rho \circ F \circ \sigma$.

The following are results about degree-one rational functions over $\F_{q^2}$ satisfying certain properties \cite{Ding4,Zieve}. These are very useful when we study geometric properties of the accompanying function $g_{\c}(X)$.
\begin{lemma}\label{mu}
A degree-one $\rho(X) \in \F_{q^2}(X)$ permutes $\mu_{q+1}$ if and only if $\rho(X)=(\bar{\beta}X+\bar{\alpha})/(\alpha X+\beta)$ for some $\alpha,\beta \in \F_{q^2}$ with $\alpha\bar{\alpha} \neq \beta\bar{\beta}$.
\end{lemma}
\begin{lemma}\label{P1}
A degree-one $\rho(X) \in \F_{q^2}(X)$ maps $\mu_{q+1}$ onto $\P^1(\F_q)$ if and only if $\rho(X)=(\delta X+\gamma\bar{\delta})/(X+\gamma)$ for some $\gamma \in \mu_{q+1}$ and $\delta \in \F_{q^2} \setminus \F_q$.
\end{lemma}
Given $D(X) \in \F_{q^2}[X]$ and $r \ge 0$, define $G_0(X):=X^rD(X)^{q-1}$. It is easy to see that $G_0(\mu_{q+1}) \subset \mu_{q+1} \cup \{0\}$. Moreover, $G_0(\mu_{q+1}) \subset \mu_{q+1}$ if and only if $0 \notin D(\mu_{q+1})$, and under this situation $G_0$ induces the same function as the rational function $G(X):=X^rD^{(q)}(1/X)/D(X)$ on $\mu_{q+1}$. While $G(X)$ is usually defined over $\F_{q^2}$, it can be transformed into a  rational function defined over $\F_q$ as follows:
\begin{lemma}\label{P2}\cite[Lemma 2.11]{Ding4}
    Let $G(X)=X^rD^{(q)}(1/X)/D(X)$ for some $D(X) \in \F_{q^2}[X]$ and $r \ge 0$. Let $\rho,\sigma \in \F_{q^2}(X)$ be any degree-one rational functions mapping $\mu_{q+1}$ onto $\P^1(\F_q)$, and define
        \[h=\rho \circ g \circ \sigma^{-1}.\]
    Then we have $h(X) \in \F_q(X)$.
\end{lemma}
\subsection{Branch points and ramification}
Here we introduce the concepts of branch points and ramification, which are the main tools to describe the geometric properties of the  accompanying rational function $g_{\c}(X)$ later on.

For a non-constant rational function $G(X) \in \overline{\F}_q(X)$, write $G(X)=N(X)/D(X)$ where $N(X), D(X) \in \overline{\F}_q[X]$ with $\gcd(N(X),D(X))=1$. For any $\alpha \in \overline{\F}_q$, define $H_\alpha(X) \in \overline{\F}_q[X]$ as
$$H_\alpha(X):=\begin{cases}
N(X)-G(\alpha)D(X) &(G(\alpha) \in \overline{\F}_q),\\
D(X) &(G(\alpha)=\infty).
\end{cases}$$
It is clear that $H_\alpha(\alpha)=0$ for all $\alpha \in \overline{\F}_q$. The \emph{ramification index} $e_G(\alpha)$ of $\alpha$ is then its multiplicity as a root of $H_\alpha(X)$. The ramification index of $\infty$ is defined as $e_G(\infty):=e_{G_1}(0)$, where $G_1(X):=G(\frac{1}{X})$. Given any $\beta \in \P^1(\overline{\F}_q)$, the \emph{ramification multiset} $E_G(\beta)$ of $G(X)$ over $\beta$ is the multiset of ramification indices $e_G(\alpha)$ for any $\alpha \in G^{-1}(\beta)$, that is, $E_G(\beta):=\left[e_G(\alpha): \alpha \in G^{-1}(\beta)\right]$. In particular, the elements in $E_G(\beta)$ are positive integers whose sum is $\deg G$. We call $\alpha \in \P^1(\overline{\F}_q)$ a \emph{ramification point} (or \emph{critical point}) of $G(X)$ if $e_G(\alpha) > 1$, and its corresponding image $G(\alpha)$ a \emph{branch point} (or \emph{critical value}) of $G$. Hence a point $\beta \in \P^1(\overline{\F}_q)$ is a branch point of $G$ if and only if $E_G(\beta) \neq [1^{\deg G}]$, where $[m^n]$ denotes the multiset consisting of $n$ copies of $m$, or equivalently, $\sharp G^{-1}(\beta) < \deg G$.

One most important result regarding ramification is known as the Hurwitz genus formula (see \cite[Corollary 3.5.6]{Stich}), which applying to the field extension $\overline{\F}_q(x)/\overline{\F}_q(G(x))$ for transcendental $x$ over $\overline{\F}_q$ yields the following result used in our proof:
\begin{lemma}\label{RH}
	Let $G(X) \in \overline{\F}_q(X)$ be a rational function of degree $n$. Then
$$2n-2 \geq \sum_{\alpha \in \P^1(\overline{\F}_q)} (e_G(\alpha)-1)$$
with equality holds if and only if $\mathrm{char}(\overline{\F}_q) \nmid e_G(\alpha)$ for all $\alpha \in \P^1(\overline{\F}_q)$.
\end{lemma}
\subsection{Linear equivalence, EA-equivalence and CCZ-equivalence}

Linear equivalence, EA-equivalence and CCZ-equivalence are equivalence relations of functions over the finite field $\F_{p^n}$ under which planar (or PN) and APN properties are invariant. Due to these equivalence relations, one PN (or APN) function can generate a huge class of PN (resp. APN) functions. While the notion of these equivalence relations was introduced in 2006 in \cite{Bud1}, the ideas behind this notion appeared much earlier \cite{Car-3,Nyb2}. Let us first recall some definitions:

\begin{definition} A function $F: \F_{p^n} \to \F_{p^n}$ is called
\begin{itemize}
\item \emph{linear} if $F(\alpha+\beta)=F(\alpha)+F(\beta)$ for any $\alpha, \beta \in \F_{p^n}$;

\item \emph{affine} if $F$ is a sum of a linear function and a constant;

\item \emph{affine permutation} (or \emph{linear permutation}) if $F$ is both affine (resp. linear) and a permutation on $\F_{p^n}$.

\item \emph{Dembowski-Ostrom polynomial} (DO polynomial) if
\[F(X)=\sum_{0 \le k,j <n} a_{k,j} X^{p^k+p^j}, \qquad a_{ij} \in \F_{p^n}. \]
\end{itemize}
\end{definition}

In particular, $F$ is affine if and only if $F(X)=b+\sum_{j=0}^{n-1} a_jX^{p^j}$ where $a_j,b \in \F_{p^n}$ for any $j$.

\begin{definition} Two functions $F$ and $F'$ from $\F_{p^n}$ to itself are called:
\begin{itemize}
\item \emph{affine equivalent} (or \emph{linear equivalent}) if $F'=A_1 \circ F \circ A_2$, where the mappings $A_1,A_2$ are affine (resp. linear) permutations of $\F_{p^n}$;
\item \emph{extended affine equivalent} (EA-equivalent) if $F'=A_1 \circ F \circ A_2+A$, where the mappings $A,A_1,A_2$ are affine, and where $A_1,A_2$ are permutations of $\F_{p^n}$;
\item \emph{Carlet-Charpin-Zinoviev equivalent} (CCZ-equivalent) if for some affine permutation $\mathcal{L}$ of $\F_{p^n}^2$ the image of the graph of $F$ is the graph of $F'$, that is, $\mathcal{L}(G_F)=G_{F'}$, where
    \[G_F=\left\{(x,F(x)): x \in \F_{p^n}\right\}, \quad G_{F'}=\left\{(x,F'(x)): x \in \F_{p^n}\right\}. \]
\end{itemize}
\end{definition}

It is obvious that linear equivalence is a particular case of affine equivalence, and affine equivalence is a particular case of EA-equivalence. It was known that EA-equivalence is a particular case of CCZ-equivalence and every permutation is CCZ-equivalence to its inverse \cite{Car-3}. CCZ-equivalence is more general than EA-equivalence but there are particular cases of functions for which CCZ-equivalence can be reduced to EA-equivalence. For instance, CCZ-equivalence coincides with EA-equivalence for planar functions \cite{Kyu}, and CCZ-equivalence coincides with linear equivalence for DO planar functions \cite{Lilya1,Lilya2}. Since the quadrinomial $f_{\c}(X)$ (\ref{DO}) is a DO polynomial, we only consider linear equivalence in this paper. For simplicity, if two such polynomials $f_1$ and $f_2$ are linear equivalent, we call them \emph{equivalent} to each other.


\subsection{Auxiliary lemmas}
In this paper we will frequently use the following result.
\begin{lemma}\cite[Lemma 2.1]{Faruk} \label{2:gcd} For a prime $p$,

i. $\gcd(p^k-1,p^{\l}-1)=p^{\gcd(k,\l)}-1$.

ii. \[\gcd(p^k+1,p^{\l}-1)=\left\{\begin{array}{ll}
1 & \frac{\l}{\gcd(k,\l)} \mbox{ is odd, and } p=2, \\
2 & \frac{\l}{\gcd(k,\l)} \mbox{ is odd, and } p \ge 3, \\
p^{\gcd(k,\l)}+1 & \frac{\l}{\gcd(k,\l)} \mbox{ is even}.
\end{array}
\right. \]
\end{lemma}

Let $\F_q(t)$ be the rational function field over $\F_q$ and let $f(t)\in \F_q(t)$ be a nonzero rational function. Write $\supp f$ for the collection of places of $\F_q(t)$ (including $\infty$) at which $f$ has either zero or pole. Let $\chi$ be a multiplicative character of $\F_q^*$ of order $d>1$. Extend $\chi$ to a function on $\P^1(\F_q)$ by defining $\chi(0)=\chi(\infty)=0$.

We have the following Weil bound (see \cite[Theorem 3, p. 94]{WLi} or \cite{Weil}):
\begin{lemma} \label{2:weilbound}
Let $\chi, f$ be defined as above, with $m$ being the total degree of the places in $\supp f$ other than $\infty$. Suppose $f$ is not a d-th power, then
\[ \left|\sum_{x \in \P^1(\F_q)} \chi (f(x)) \right| \le (m-1) \sqrt{q}.\]
\end{lemma}

\section{Geometric properties of $g(X)$}\label{Cla}
For any $\c=(c_0,c_1,c_2,c_3) \in \F_{q^2}^4$, define
\begin{eqnarray} \label{3:acx} A(X)=A_{\c}(X)&:=&c_0X^{Q+1}+c_1X^Q+c_2X+c_3.\end{eqnarray}
The quadrinomial $f_{\c}(X)$ given in (\ref{DO}) can be written as $f_{\c}(X)=X^{Q+1}A\left(X^{q-1}\right)$.

Let us assume that $\c \ne \underline{0}$. Denote
\begin{eqnarray}\label{3:bcx}
B(X)=B_{\c}(X)&:=&\bar{c}_3X^{Q+1}+\bar{c}_2X^Q+\bar{c}_1X+\bar{c}_0,\\
\label{3:gc} g(X)=g_{\c}(X)&:=&B_{\c}(X)/A_{\c}(X).
\end{eqnarray}
Here we use the notation $\bar{c}:=c^q$ for any $c \in\F_{q^2}$.

Geometric properties of $g_{\c}(X)$ for a general $q$ including $q$ odd were essentially presented in \cite{Ding4}. However, the focus of \cite{Ding4} is on the question of when $f_{\c}(X)$ is a permutation on $\F_{q^2}$ which turns out to require that $q$ is even, and to address this question the complete and detailed geometric properties of $g_{\c}(X)$ were not specified explicitly especially when $q$ is odd as they were not needed in \cite{Ding4}. To prove Theorem \ref{Main}, however, we need such more detailed information. For this reason, in this section, we follow the ideas of \cite{Ding4} closely and give detailed geometric properties of $g_{\c}(X)$ for all $q$ while focusing on the case that $q$ is odd. This will be useful in classifying $g_{\c}(X)$ in the next section. To describe the result, let us define a few parameters
\begin{eqnarray}\label{3:nota}
\left\{
\begin{array}{lll}e_1&:=&c_0\bar{c}_0-c_1\bar{c}_1-c_2\bar{c}_2+c_3\bar{c}_3, \\
e_2&:=&-c_0\bar{c}_0-c_1\bar{c}_1+c_2\bar{c}_2+c_3\bar{c}_3, \\
e_3&:=&-c_0\bar{c}_0+c_1\bar{c}_1-c_2\bar{c}_2+c_3\bar{c}_3, \\
\te_2&:=&\bar{c}_2c_3-\bar{c}_0c_1,\\
\te_3&:=&\bar{c}_1c_3-\bar{c}_0c_2,\\
\te_1^2&:=&e_2^2-4\te_2\bar{\te}_2,
\end{array} \right.
\end{eqnarray}
and three polynomials
\begin{eqnarray} \label{3:WUV}
\left\{
\begin{array}{lll}
W(X)&:=&(c_1c_2-c_0c_3)X^2+e_1X+\left(\overline{c}_1\overline{c}_2-
\overline{c}_0\overline{c}_3\right),\\
U(X)&:=&\bar{\te}_2X^2+e_2X+\te_2,\\
V(X)&:=&\bar{\te}_3^{1/Q}X^2+e_3^{1/Q}X+\te_3^{1/Q}.
\end{array} \right.
\end{eqnarray}
Here for simplicity, we drop the subscript $\c$ in the notation when there is no ambiguity. It is easy to see that $e_1,e_2,e_3,\te_1^2 \in \F_{q}$, $\te_1,\te_2,\te_3 \in \F_{q^2}$ and $W(X),U(X),V(X) \in \F_{q^2}[X]$ are SCR polynomials (if their degrees are two).

We first state some properties that relate the three polynomials $U(X), V(X)$ and $W(X)$ in (\ref{3:WUV}) with $A(X)$ and $B(X)$ given in (\ref{3:acx}) and (\ref{3:bcx}) respectively:
\begin{lemma}\label{UVW}
Assume $\c \neq \underline{0}$ and $q$ is odd. Denote $C(X)=\gcd(A(X),B(X))$, the greatest monic common divisor of $A(X)$ and $B(X)$.
\begin{enumerate}
    \item If $U(X)$ and $V(X)$ are not both zero, then $C(X) \mid \gcd(U(X),V(X))$ and $C(X)$ is either $X$ or a monic \emph{SCR} polynomial of degree at most two;
    \item Assume none of $U(X), V(X)$ or $W(X)$ is the zero polynomial. Denote $\Gamma=\Gamma_{\c}$ the union of the set of roots of $V(X)$ and the set with $(2-\deg V)$ copies of $\infty$ (this set is either empty if $\deg V=2$, or $\{\infty\}$ if $\deg V=1$); denote $\Lambda=\Lambda_{\c}$ the union of the set of roots of $W(X)$ and the set with $(2-\deg W)$ copies of $\infty$.
    \begin{enumerate}
    \item $\sharp \Gamma=\sharp \Lambda \in \{1,2\}$, and the cardinality is $1$ if and only if $\te_1=0$;
    \item Either both $\Gamma,\Lambda \subset \mu_{q+1}$, or both sets are of the form $\left\{\alpha,\bar{\alpha}^{-1}\right\}$ for some $\alpha \in \F_{q^2} \setminus \mu_{q+1}$;
    \item $\Gamma$ is the complete set of ramification points of $g$, and any branch point of $g$ is in $\Lambda$;
    \item Assume $C(X)=1$. Then $\Lambda$ gives the complete set of branch points of $g$. Moreover, the $g$-ramification multiset of any $\lambda \in \Lambda$ is either $[Q+1]$ or $[1,Q]$.
    \end{enumerate}
\end{enumerate}
\end{lemma}

We remark when $q$ is even, all the above statements of Lemma \ref{UVW} remain true except 2 c) in the special case that $Q=2$ and $C(X) \neq 1$ and $\deg g=1$, in this case we have $\Gamma =\emptyset$ as $g$ has no ramification or branch points.

We also remark that the essence of Lemma \ref{UVW} is essentially contained in the proof of \cite[Theorem 3.1]{Ding4}. For the sake of completeness, we include the proof here. We focus on the case that $q$ is odd.

\begin{proof}

For any polynomial of the form $P(X):=\alpha X^2+\beta X+\gamma$ with $\alpha,\beta,\gamma \in \F_{q^2}$, define
\[\Delta(P):=\beta^2-4 \alpha \gamma. \]
Thus if $\deg (P)=2$ then $\Delta(P)$ is the discriminant of $P(X)$. Recall from \cite{Ding4} that $U(X), V(X)$ and $W(X)$ stated in (\ref{3:WUV}) satisfy the following identities:
\begin{align}
        U(X)&=(\bar{c}_3X+\bar{c}_2)A(X)-(c_0X+c_1)B(X);\label{E1}\\
        V(X)^Q&=A(X)B'(X)-A'(X)B(X);\label{E2}\\
        \Delta(W)&=\Delta(U)=\Delta(V)^Q=\te_1^2;\label{E3}\\
        U(X)V(X)^Q&=W(g(X))A(X)^2.\label{E4}
\end{align}
It is immediate from (\ref{E1}) and (\ref{E2}) that $C(X) \mid U(X)$ and $C(X) \mid V(X)^Q$. Moreover, the second part of Statement 1 easily follows from the fact that each of $U(X)$ and $V(X)$ is either a constant times $X$ or a degree-two SCR polynomial and the assumption that they are not both zero. Now assume $C(X) \nmid V(X)$. This implies $C(X)$ is not square-free. In particular it must be the square of a linear SCR polynomial. Hence $U(X)$ is a constant multiple of $C(X)$, which implies $\te_1=0$. Then Equation (\ref{E3}) implies $\Delta(V)=0$, and $C(X) \mid V(X)^Q$ implies that the unique root of $C(X)$ is a root of $V(X)$, so $V(X)$ is also a constant multiple of $C(X)$, a contradiction. Hence we have $C(X) \mid \gcd(U(X),V(X))$.

    Statements 2 a) and 2 b) follow from Equation (\ref{E3}), Lemmas \ref{SCR} and \ref{lem1} in Section \ref{Pre}. In addition, the right hand side of (\ref{E2}) is simply $A(X)^2g'(X)$. Hence we see that any ramification point of $g$ in $\overline{\F}_q$ must be in $\Gamma$. On the other hand, writing $A(X)=A_0(X)C(X)$ and $B(X)=B_0(X)C(X)$ for $A_0, B_0 \in \F_{q^2}[X]$, so that $\gcd(A_0(X),B_0(X))=1$. Then (\ref{E4}) can be rewritten as
    \begin{equation}\label{E5}
    U(X)V(X)^Q=W(g(X))A_0(X)^2C(X)^2.
    \end{equation}
Here we note that $W(g(X))A_0(X)^2=W\left(\frac{B_0(X)}{A_0(X)}\right)A_0(X)^2$ is a polynomial in $\F_{q^2}$.

Denote $\Gamma_1:=\Gamma \setminus \{\infty\}$. If $\te_1=0$, then $U(X)=\bar{\te}_2(X-\alpha)^2$, $V(X)=\bar{\te}_3^{1/Q}(X-\gamma)^2$ and $W(X)=c(X-\lambda)^2$ for some $c \in \F_{q^2}^*$ and $\alpha,\gamma,\lambda \in \mu_{q+1}$. We have $\Gamma_1=\{\gamma\}$ and $\Lambda=\{\lambda\}$. Putting these into (\ref{E5}) and taking square root on both sides, we have
$$(X-\alpha)(X-\gamma)^Q=\tilde{c}\,(B_0(X)-\lambda A_0(X))C(X)$$
for some $\tilde{c} \in \F_{q^2}^*$.

Here $C(X)=1$ if $\alpha \neq \gamma$, otherwise $C(X)=(X-\gamma)^i$ for some $i \leq 2$. Upon dividing both sides by $C(X)$, it is easy to see that $\gamma$ is a multiple root of the LHS (the factor $X-\gamma$ appears with exponent at least $Q-1 \geq 2$). Hence it is also so in the RHS, which is precisely equivalent to saying that $\gamma \in \Gamma_1$ is a ramification point of $g(X)$, and its image is a branch point of $g$, that is, $\lambda \in \Lambda$.

Now assume $\te_1 \neq 0$. Then $\sharp \Gamma=\sharp \Lambda=2$ and $U(X), V(X)$ and $W(X)$ are square-free. Writing $U(X)=U_0(X)C(X)$ and $V(X)=V_0(X)C(X)$, and upon dividing both sides of (\ref{E5}) by $C(X)^2$, we obtain
$$U_0(X)V_0(X)^QC(X)^{Q-1}=W(g(X))A_0(X)^2.$$
Now any $\gamma \in \Gamma_1$ is a root of either $V_0(X)$ or $C(X)$. Since $Q \geq 3$, $X-\gamma$ appears as a factor with exponent at least two in LHS. Hence $\gamma$ is a multiple root of LHS and hence also a root of RHS. This clearly implies $\gamma$ is a ramification point of $g(X)$, whose image $g(\gamma)$ is either a root of $W(X)$ or $\infty$, and is hence a branch point of $g$. Note that $g(\gamma)=\infty$ if and only if $\gamma$ is a root of $A_0(X)$, whence this may happen only if $\deg W=1$. In either case, we see that the branch point $g(\gamma) \in \Lambda$.

To complete the proof of Statement 2 c), we need to show that $\infty$ is a ramification point of $g$ if and only if $\deg V=1$. This can be proved as follows: $\infty$ is a ramification point of $g$ if and only if 0 is a ramification point of $g(1/X)=g_{\c'}(X)$ where $\c'=(c_3,c_2,c_1,c_0)$. It is easy to verify that $V_{\c'}(X)=-V^{(q)}(X)$ and $W_{\c'}(X)=W(X)$. Hence 0 is a ramification point of $g_{\c'}$ if and only if 0 is a root of $V^{(q)}$ if and only if 0 is a root of $V$, which is equivalent to $\deg V=1$. In this case, the branch point $g(\infty)=g_{\c'}(0)$ is in $\Lambda_{\c'}=\Lambda$.

Finally, assume $C(X)=1$. Then $\deg g=\max\{\deg A, \deg B\}$. Note that the latter is not $Q+1$ if and only if $c_0=c_3=0$, whence in this case $C(X) \neq 1$. Hence $\deg g=Q+1$. By (\ref{E4}), for any ramification point $\gamma \in \Gamma, e_g(\gamma)$ is either $Q+1$ or $Q$ according to whether it is a root of $U(X)$ or not. This implies that the $g$-ramification multiset of the corresponding branch point $g(\gamma) \in \Lambda$ is $[Q+1]$ or $[1,Q]$. Hence each branch point corresponds to exactly one ramification point. By Statements 2 a) and c), then $g(\Gamma)=\Lambda$, so $\Lambda$ gives the complete set of branch points of $g$.
\end{proof}

Armed with Lemma \ref{UVW}, we can give more detailed information about geometric properties of $g_{\c}(X)$. We remark that Lemma \ref{lem11} are still true when $p=2$, though we focus on the case that $q$ is odd.

\begin{lemma}\label{lem11}
Assume $\c \neq \underline{0}$ and $q$ is odd. Then we have
\begin{enumerate}
    \item $g(X)$ is constant if and only if $U(X)$ and $V(X)$ are both zero;
    \item $g(X)$ is non-constant and $A(X)$ has a root in $\mu_{q+1}$ if and only if at least one of $U(X), V(X)$ is a nonzero polynomial with roots in $\mu_{q+1}$ and $\deg g \neq Q+1$;
    \item $g(X)$ is non-constant and $A(X)$ has no roots in $\mu_{q+1}$ if and only if one of the following is true:
    \begin{enumerate}
        \item $g(X)$ is $\overline{\F}_q$-linear equivalent to $X^n$ where $n \in \{Q+1, Q-1\}$; this occurs if and only $U(X)/V(X) \in \F_{q^2}^*$; or
        \item $g(X)$ has at least one branch point in $\P^1(\overline{\F}_q)$ with ramification multiset $[1,Q]$.
    \end{enumerate}
\end{enumerate}
\end{lemma}

We remark that lemma \ref{lem11} is essentially contained in the proof of \cite[Theorem 3.1]{Ding4}. In particular, 1) and 2) correspond to (A3) and (A4) of \cite[Theorem 3.1]{Ding4} respectively with explicit conditions on $U(X),V(X)$ which can be obtained following the proof of \cite[Theorem 3.1]{Ding4}; 3 a) and 3 b) correspond to (A2) and (A1) of \cite[Theorem 3.1]{Ding4} with explicit conditions which can be found following the proof of \cite[Theorem 3.1]{Ding4}. For the sake of completeness, we provide a detailed proof.

\begin{proof}

For 1), if $g(X)$ is constant, then there is $\lambda \in \F_{q^2}^*$ such that $\bar{c}_3=\lambda c_0, \bar{c}_2=\lambda c_1, \bar{c}_1=\lambda c_2$ and $\bar{c}_0=\lambda c_3$. Putting into the formulas in (\ref{3:nota}) implies $e_2=e_3=\te_2=\te_3=0$. Hence $U(X)$ and $V(X)$ are both zero; Next assume that both $U(X)$ and $V(X)$ are zero. By using (\ref{E1}), we see that $g(X)=\frac{\bar{c}_3X+\bar{c}_2}{c_0X+c_1}$, so $\deg g \leq 1$. By using (\ref{E2}), we see that $g'(X)=V(X)^QA(X)^2=0$, so $g(X)$ is non-separable. Combining both, then $g(X)$ must be constant. So we have proved 1) of Lemma \ref{lem11}.

For 2), let us first assume that $g(X)$ is non-constant and $A(X)$ has a root $\alpha$ in $\mu_{q+1}$. Then at least one of $U(X)$ and $V(X)$ is nonzero. Note that the multiset of roots of $B(X)$ in $\overline{\F}_q^*$ is the same as the multiset of $(-q)$-th powers of roots of $A(X)$ in $\overline{\F}_q^*$. Hence $\alpha=\alpha^{-q}$ is also a root of $B(X)$, which implies $C(X)=\gcd(A(X),B(X))$ has a root in $\mu_{q+1}$. On the other hand, if $C(X)$ has a root in $\mu_{q+1}$, then obviously so does $A(X)$. Hence to prove Statement 2), it suffices to show that $C(X)$ has a root in $\mu_{q+1}$ if and only if any nonzero member among $U(X), V(X)$ have roots in $\mu_{q+1}$ and $\deg g \neq Q+1$. This follows from Statement 1 of Lemma \ref{UVW} by simply noting that whenever $U(X)$ (resp. $V(X)$) is nonzero, then either all its roots are in $\mu_{q+1}$ or none of the roots are, and also that $\deg g \neq Q+1$ if and only if $C(X) \neq 1$. This proves 2) of Lemma \ref{lem11}.

As for 3), let us assume that $g(X)$ is non-constant, and $A(X)$ has no roots in $\mu_{q+1}$. We first know by 1) of Lemma \ref{lem11} that at least one of $U(X)$ and $V(X)$ is nonzero.

If $U(X)=0$ and $V(X) \neq 0$, then $\te_1=0$. By (\ref{E5}), $\Delta(V)=0$. By Lemma \ref{lem1}, $V(X)$ has a multiple root in $\mu_{q+1}$. Moreover, $U(X)=0$ implies that $\deg g \leq 1 < Q+1$ and so $A(X)$ has a root in $\mu_{q+1}$ by Statement 2) of Lemma \ref{lem11}.

If $V(X)=0$ and $U(X) \neq 0$, then $\Delta(V)=0$. By (\ref{E5}), $\te_1=0$. By Lemma \ref{lem1}, $U(X)$ has a multiple root in $\mu_{q+1}$. Moreover, $V(X)=0$ implies that $g$ is non-separable, so that $\deg g \neq Q+1$. Hence $A(X)$ has a root in $\mu_{q+1}$ by Statement 2) of Lemma \ref{lem11}.

From now on we assume both $U(X)$ and $V(X)$ are nonzero, so that $\max\{\deg U, \deg V\} \leq 2$, and so $\deg C \leq 2$ by Statement 1 of Lemma \ref{UVW}. If $\deg g=Q+1$ then $C(X)=1$ and Statement 2) of Lemma \ref{lem11} implies $A(X)$ has no roots in $\mu_{q+1}$. Moreover, by Statement 2) of Lemma \ref{UVW}, there are one or two branch points of $g$, each with ramification multiset $[Q+1]$ or $[1,Q]$. If at least one has ramification multiset $[1,Q]$ then we are in the situation of Statement 3 b). Moreover, from (\ref{E5}), we easily see that there is at least one root of $U$ that is not a root of $V$, which means $U(X)/V(X) \notin \F_{q^2}$. Now assume all branch points of $g$ have ramification multiset $[Q+1]$. In particular $p \nmid e_g(\alpha)$ for all $\alpha \in \P^1(\overline{\F}_q)$. By the Hurwitz genus formula (Lemma \ref{RH}), we see that there must be exactly two such branch points. Both have a unique $g$-preimage in $\P^1(\overline{\F}_q)$. Hence $g(X)$ is $\overline{\F}_q$-linear equivalent to $X^{Q+1}$.

    Finally assume $\deg g \neq Q+1$, so that $C(X) \neq 1$. By Statement 2) of Lemma \ref{lem11}, then $V(X)$ (and $U(X)$) has no roots in $\mu_{q+1}$. Hence the set $\Gamma$ as defined in Lemma \ref{UVW} is of the form $\{\alpha,\bar{\alpha}^{-1}\}$ for some $\alpha \in \F_{q^2} \setminus \mu_{q+1}$. This could happen only if $C(X)$ is a nonzero constant multiple of $V(X)$ (and hence also $U(X)$). If $C(X)=X$, then $c_0=c_3=0$ while at least one of $c_1$ or $c_2$ is nonzero, so $\max\{\deg A, \deg B\}=Q$. If $\deg C=2$, then at least one of $c_0$ or $c_3$ is nonzero, so $\max\{\deg A,\deg B\}=Q+1$. In either case we have $\deg g=Q-1$. Putting these into (\ref{E5}) and dividing both sides by $C(X)^2$ we have
    $$V(X)^{Q-1}=\tilde{c} \, W(g(X))A_0(X)^2$$
    for some constant $\tilde{c} \in \F_{q^2}^*$.

    This immediately implies that $e_g(\gamma)=Q-1=\deg g$ for all $\gamma \in \Gamma$. Since $Q \geq 3$, $g$ has two branch points with unique preimages in $\P^1(\overline{\F}_q)$, and thereby is $\overline{\F}_q$-linear equivalent to $X^{Q-1}$.

    Finally, if $g(X)$ is $\overline{\F}_q$-linear equivalent to $X^{Q+1}$, then $C(X)=1$, by counting multiplicity of the term $X-\gamma$ for all $\gamma \in \Gamma$ on both sides of (\ref{E5}), we see that $U(X)/V(X) \in \F_{q^2}^*$; if $g(X)$ is $\overline{\F}_q$-linear equivalent to $X^{Q-1}$, then $\deg C=2$, we also have $U(X)/V(X) \in \F_{q^2}^*$.

    Now the proof of Lemma \ref{lem11} is complete.
\end{proof}

\section{Classification of $g(X)$}\label{Cla-2}
We will see in later sections that Cases 1) and 2) of Lemma \ref{lem11} do not yield planar functions. So we examine $g(X)$ further according to Case 3) of Lemma \ref{lem11}. We first consider Case 3) a).

\begin{lemma}\label{lem5}
Let $\c \ne \underline{0}$. Assume that $A(X)$ has no roots in $\mu_{q+1}$ and $g(X)$ is $\overline{\F}_q$-linear equivalent to $X^n$ where $n \in \{Q+1,Q-1\}$. Then one of the following holds:
\begin{enumerate}
\item $g(X)=\rho^{-1} \circ X^{Q+1} \circ \sigma$ for some degree-one $\rho, \sigma \in \F_{q^2}(X)$ both of which map $\mu_{q+1}$ onto $\P^1(\F_q)$;

\item $g(X)=\rho^{-1} \circ X^{Q+1} \circ \sigma$ for some degree-one $\rho, \sigma \in \F_{q^2}(X)$ both of which permute $\mu_{q+1}$;

\item $g(X)=\rho^{-1} \circ X^{Q-1} \circ \sigma$ for some degree-one $\rho, \sigma \in \F_{q^2}(X)$ both of which permute $\mu_{q+1}$, and there exists $\alpha \in \F_{q^2} \setminus \mu_{q+1}$ such that $\sigma(\{\alpha,\bar{\alpha}^{-1}\})=\{0,\infty\}$ and $\gcd(A(X),B(X))$ is a constant multiple of $(\bar{\alpha}X-1)(X-\alpha)$.
\end{enumerate}
\end{lemma}
We remark that Lemma \ref{lem5} is essentially the same as \cite[Proposition 5.3 (2)]{Ding4}, the only difference is that in \cite[Proposition 5.3 (2)]{Ding4}, there was an extra condition that $g(X)$ permutes $\mu_{q+1}$, but the proof is essentially the same. We omit the proofs.


Next we consider $g(X)$ in Case 3 b) of Lemma \ref{lem11}. This is just (A1) of \cite[Theorem 3.1]{Ding4}. As was shown in \cite{Ding4}, the $g(X)$ appearing in Lemma \ref{lem6} does not yield permutations, which is why such $g(X)$ was not treated in \cite{Ding4}. However for the purpose of our paper, we do need to study this case more deeply. We state the result as follows.
\begin{lemma}\label{lem6}
   Let $\c \ne \underline{0}$. Assume $g(X)$ has at least one branch point with ramification multiset $[1,Q]$. Then either one of the following holds:
\begin{enumerate}
    \item $g(X)=\rho^{-1} \circ \frac{X^{Q+1}}{X+1} \circ \sigma$ for some degree-one $\rho, \sigma \in \F_{q^2}(X)$ both of which map $\mu_{q+1}$ to $\P^1(\F_q)$;

   \item $g(X)=\rho^{-1} \circ \frac{X^Q}{X^{Q+1}+\ep} \circ \sigma$ for some degree-one $\rho, \sigma \in \F_{q^2}(X)$ both of which map $\mu_{q+1}$ to $\P^1(\F_q)$, and $\ep \in \F_q^*$;

   \item $g(X)=\rho^{-1} \circ \frac{X^Q(X-1)}{X+\ep} \circ \sigma$ for some degree-one $\rho, \sigma \in \F_{q^2}(X)$ both of which map $\mu_{q+1}$ to $\P^1(\F_q)$, and $\ep \in \F_q^* \setminus \{-1\}$;

   \item $g(X)=\rho^{-1} \circ \frac{X^Q(\bar{\ep} X+1)}{X+\ep} \circ \sigma$ for some degree-one $\rho, \sigma \in \F_{q^2}(X)$ both of which permute $\mu_{q+1}$, and $\ep \in \F_{q^2}^* \setminus \mu_{q+1}$.
\end{enumerate}
\end{lemma}
To prove Lemma \ref{lem6}, we essentially use the ideas in the proof of \cite[Theorem 3.1]{Ding4}.

\begin{proof}
Since $\deg g=Q+1$, from Lemma \ref{lem11}, we see that $C(X)=1$, $A(X)$ has no roots in $\mu_{q+1}$, $U(X)$ and $V(X)$ are both nonzero, $U(X)/V(X) \notin \F_{q^2}^*$, and $g(X)$ is not $\overline{\F}_q$-linear equivalent to a monomial.

We first consider the case where $\gcd(U,V) \neq 1$. Since $U(X)/V(X) \notin \F_{q^2}^*$, we must have $\deg U=\deg V=2$. In addition, $U(X)$ and $V(X)$ only have one common root, say $\alpha$. If $\alpha$ is not in $\mu_{q+1}$, then $\alpha^{-q}$ will also be a common root of $U(X)$ and $V(X)$ since they are SCR polynomials, a contradiction. Hence $\alpha \in \mu_{q+1}$. It cannot be a multiple root either. Hence $\te_1 \in \F_{q^2} \setminus \F_q$ by Lemma \ref{lem1}, which in turn implies that $U(X),V(X)$ and $W(X)$ all have two distinct roots in $\mu_{q+1}$ by (\ref{E3}). Let $\beta_1$ and $\beta_2$ be the other roots of $U(X)$ and $V(X)$ respectively. In addition, let $\Lambda=\{\gamma_1,\gamma_2\}$ be the set of roots of $W$. Using Equation (\ref{E5}), we see that $\alpha$ has multiplicity $Q+1$ as a ramification point of $g$, while $\beta_2$ has multiplicity $Q$. Lemma \ref{UVW} implies that $\alpha$ is the unique $g$-preimage of some element in $\Lambda$. Let us assume that this root is $\gamma_2$. Then $\gamma_1$ is the unique branch point with $g$-ramification multiset $[1,Q]$. In fact Equation (\ref{E5}) implies that $g(\beta_1)=\gamma_1=g(\beta_2)$. Define $$\tilde{\sigma}(X):=\frac{\delta_1(X-\alpha)}{X-\beta_2}$$
and
$$\tilde{\rho}(X):=\frac{\delta_2(X-\gamma_2)}{X-\gamma_1},$$
where $\delta_1,\delta_2 \in \F_{q^2}^*$ such that $\bar{\delta_1}/\delta_1=\alpha/\beta_2$ and $\bar{\delta_2}/\delta_2=\gamma_2/\gamma_1$ respectively. By Lemma \ref{P1}, $\tilde{\sigma},\tilde{\rho}$ both map $\mu_{q+1}$ onto $\P^1(\F_q)$. In addition, $\tilde{\sigma}(\beta_2)=\infty=\tilde{\rho}(\gamma_1)$ and $\tilde{\sigma}(\alpha)=0=\tilde{\rho}(\gamma_2)$. Combining these results, we find that the rational function $h(X):=\tilde{\rho} \circ g \circ \tilde{\sigma}^{-1}(X)$ has degree $Q+1$ and maps $\P^1(\F_q)$ into $\P^1(\F_q)$, and 0 is the unique $h$-preimage of 0, and $\infty$ is another branch point of $h$, with $\infty$ as an $h$-preimage of multiplicity $Q$. Together with Lemma \ref{P2}, we conclude that $h(X)=\frac{\lambda X^{Q+1}}{X+\ep}$ for some $\lambda, \ep \in \F_q^*$. This implies $g(X)=\tilde{\rho}^{-1} \circ h \circ \tilde{\sigma}(X)=\rho^{-1} \circ \frac{X^{Q+1}}{X+1} \circ \sigma(X)$, where $\sigma(X):=\ep^{-1}\tilde{\sigma}(X)$ and $\rho(X):=\lambda^{-1}\ep^{-Q}\tilde{\rho}(X)$ are both of degree-one in $\F_{q^2}(X)$ and map $\mu_{q+1}$ onto $\P^1(\F_q)$. This proves 1) of Lemma \ref{lem6}.

From now on we assume $\gcd(U,V)=1$. We first assume $U(X)$ has a root in $\mu_{q+1}$. If this is the unique root of $U$, then $\te_1=0$ by Lemma \ref{lem1}, and $V(X)$ and $W(X)$ both have a unique multiple root in $\mu_{q+1}$ by Equation (\ref{E3}) too. Let us denote by $\alpha_1,\alpha_2$ and $\beta$ the unique roots of $U(X),V(X)$ and $W(X)$ respectively. Since $\gcd(U,V)=1$, $\alpha_1 \neq \alpha_2$. Since $W$ has only one root, by Lemma \ref{UVW}, we must have $g(\alpha_1)=\beta=g(\alpha_2)$, and $\beta$ is the unique branch point of $g$, whose ramification multiset is $[1,Q]$. Now define
$$\sigma(X):=\frac{\gamma(X-\alpha_1)}{X-\alpha_2}$$
and
$$\tilde{\rho}(X):=\frac{X-\beta}{\delta X-\beta\bar{\delta}},$$
where $\gamma,\delta \in \F_{q^2} \setminus \F_q$, with $\bar{\gamma}/\gamma=\alpha_2/\alpha_1$. Then both $\sigma$ and $\tilde{\rho}$ map $\mu_{q+1}$ onto $\P^1(\F_q)$. In addition, $\sigma(\alpha_1)=\infty$ and $\sigma(\alpha_2)=0=\tilde{\rho}(\beta)$. Combining these we find that the rational function $h(X):=\tilde{\rho} \circ g \circ \sigma^{-1}(X)$ has degree $Q+1$ and maps $\P^1(\F_q)$ into $\P^1(\F_q)$, 0 is the unique branch point of $h$, with 0 and $\infty$ as its $h$-preimages of multiplicity $Q$ and 1 respectively. By Lemma \ref{P2}, $h(X) \in \F_q(X)$. Hence $h(X)=\frac{X^Q}{D(X)}$ for some $D(X) \in \F_q[X]$ of degree $Q+1$ with $D(0) \neq 0$. Since $g(X)=\frac{B(X)}{A(X)}$, we must have $D(X)=A_{\c_0}(X)$ for some $\c_0=(c_{00},c_{01},c_{02},c_{03}) \in \F_q^4$, with $c_{00}c_{03} \neq 0$. Then $D(X)^2h'(X)=-X^Q(c_{00}X^Q+c_{02})$. Since 0 is the unique critical point of $h$, we must have $c_{02}=0$. Hence $h(X)=\frac{X^Q}{c_{00}X^{Q+1}+c_{01}X^Q+c_{03}}$. This implies $g(X)=\tilde{\rho}^{-1} \circ h \circ \sigma(X)=\rho^{-1} \circ \frac{X^Q}{X^{Q+1}+\ep} \circ \sigma(X)$, where $\rho(X)=(c_{00}\tilde{\rho}(X))/(1-c_{01}\tilde{\rho}(X))$ and $\ep=c_{03}/c_{00} \in \F_q^*$, so $\rho(\mu_{q+1})=\P^1(\F_q)$. This proves 2) of Lemma \ref{lem6}.

Now we assume $U$ has two distinct roots in $\mu_{q+1}$. By Lemma \ref{lem1}, this implies $\te_1 \in \F_{q^2} \setminus \F_q$, and $V(X),W(X)$ have two distinct roots in $\mu_{q+1}$ too by Equation (\ref{E3}). Denote by $\Sigma=\{\alpha_1,\alpha_2\}, \Gamma=\{\beta_1,\beta_2\}$ and $\Lambda=\{\gamma_1,\gamma_2\}$ the set of roots of $U(X),V(X)$ and $W(X)$ respectively. Since $\gcd(U,V)=1$, we have $\Sigma \cap \Gamma=\emptyset$. By Lemma \ref{UVW}, $\Gamma$ is the set of ramification points of $g$, each of which has multiplicity $Q$. In addition, we may assume $g(\alpha_i)=\gamma_i=g(\beta_i)$ for $i=1,2$. Then both elements of $\Lambda$ have $g$-ramification multiset $[1,Q]$. Now define
$$\tilde{\sigma}(X):=\frac{\delta_1(X-\beta_2)}{X-\beta_1}$$
and
$$\tilde{\rho}(X):=\frac{\delta_2(X-\gamma_2)}{X-\gamma_1},$$
where $\delta_1,\delta_2 \in \F_{q^2}^*$ such that $\bar{\delta}_1/\delta_1=\beta_2/\beta_1$ and $\bar{\delta}_2/\delta_2=\gamma_2/\gamma_1$ respectively. Then $\tilde{\sigma},\tilde{\rho}$ both map $\mu_{q+1}$ onto $\P^1(\F_q)$. In addition, $\tilde{\sigma}(\beta_1)=\infty=\tilde{\rho}(\gamma_1)$ and $\tilde{\sigma}(\beta_2)=0=\tilde{\rho}(\gamma_2)$. Combining these we find that the rational function $h(X):=\tilde{\rho} \circ g \circ \tilde{\sigma}^{-1}(X)$ has degree $Q+1$, maps $\P^1(\F_q)$ into $\P^1(\F_q)$, 0 and $\infty$ are the branch points of $h$, as $h$-preimages of 0 and $\infty$ of multiplicity $Q$ respectively. Together with Lemma \ref{P2}, we have $h(X)=\frac{\lambda X^Q(X+\ep_1)}{X+\ep_2}$ for some $\lambda, \ep_1, \ep_2 \in \F_q^*$, with $\ep_1 \neq \ep_2$. This implies that $g(X)=\tilde{\rho}^{-1} \circ h \circ \tilde{\sigma}(X)=\rho^{-1} \circ \frac{X^Q(X-1)}{X+\ep} \circ \sigma(X)$, where $\sigma(X):=-\ep_1^{-1}\tilde{\sigma}(X)$ and $\rho(X):=-\lambda^{-1}\ep_1^{-Q}\tilde{\rho}(X)$ are both of degree-one in $\F_{q^2}(X)$ and map $\mu_{q+1}$ onto $\P^1(\F_q)$, and $\ep=-\ep_2/\ep_1 \in \F_q^* \setminus \{-1\}$. This proves 3) of Lemma \ref{lem6}.

Finally we assume $U$ has no roots in $\mu_{q+1}$. By Lemma \ref{lem1}, this implies $\te_1 \in \F_q^*$, and $V(X),W(X)$ both have no roots in $\mu_{q+1}$ by Equation (\ref{E3}). If $\deg U=1$, define $\Sigma=\{0,\infty\}$. If $\deg U=2$, define $\Sigma$ as the set of roots of $U(X)$, which is of the form $\{\alpha,\bar{\alpha}^{-1}\}$ for some $\alpha \in \F_{q^2}^* \setminus \mu_{q+1}$. To combine these two cases, let us simply write $\Sigma=\{\alpha,\bar{\alpha}^{-1}\}$ where $\alpha \in \P^1(\F_{q^2}) \setminus \mu_{q+1}$. Assume $\Gamma=\{\beta,\bar{\beta}^{-1}\}$ and $\Lambda=\{\gamma,\bar{\gamma}^{-1}\}$ as defined in Lemma \ref{UVW} respectively, where $\beta,\gamma \in \P^1(\F_{q^2}) \setminus \mu_{q+1}$. Since $\gcd(U,V)=1$, we have $\Sigma \cap \Gamma=\emptyset$. Hence by Lemma \ref{UVW}, $\Gamma$ is the set of ramification points of $g$, each of which has multiplicity $Q$. In addition, we may assume $g(\alpha)=\gamma=g(\beta)$. Since $g(\bar{X}^{-1})=\overline{g(X)}^{-1}$, we automatically have $g(\bar{\alpha}^{-1})=\bar{\gamma}^{-1}=g(\bar{\beta}^{-1})$. Then both elements of $\Lambda$ have $g$-ramification multiset $[1,Q]$. Now define $\sigma(X):=-\frac{\bar{\beta}X-1}{X-\beta}$ if $\beta \in \F_{q^2} \setminus \mu_{q+1}$ and $\sigma(X):=X$ if $\beta=\infty$. Similarly, define $\tilde{\rho}(X)=-\frac{\bar{\gamma}X-1}{X-\gamma}$ if $\gamma \in \F_{q^2} \setminus \mu_{q+1}$ and $\tilde{\rho}(X)=X$ if $\gamma=\infty$. In all these cases $\sigma,\tilde{\rho}$ both permute $\mu_{q+1}$. In addition, $\sigma(\beta)=\infty=\tilde{\rho}(\gamma)$ and $\sigma(\bar{\beta}^{-1})=0=\tilde{\rho}(\bar{\gamma}^{-1})$. Combining these we find that the rational function $h(X):=\tilde{\rho} \circ g \circ \sigma^{-1}(X) \in \F_{q^2}(X)$ has degree $Q+1$, maps $\mu_{q+1}$ into $\mu_{q+1}$, 0 and $\infty$ are the branch points of $h$, as $h$-preimages of 0 and $\infty$ of multiplicity $Q$ respectively. Hence $h(X)=X^Qh_1(X)$ for some degree-one $h_1(X) \in \F_{q^2}(X)$ such that $h_1(0), h_1(\infty) \notin \{0,\infty\}$. Since both $h(X)$ and $X^Q$ map $\mu_{q+1}$ to $\mu_{q+1}$, so is $h_1(X)$. Therefore $h_1(X)$ permutes $\mu_{q+1}$ and by Lemma \ref{mu} it must be of the form $\frac{\bar{\ep}_2X+\bar{\ep}_1}{\ep_1X+\ep_2}$ for some $\ep_1,\ep_2 \in \F_{q^2}^*$, such that $\ep_1\bar{\ep}_1 \neq \ep_2\bar{\ep}_2$. This implies $g(X)=\tilde{\rho}^{-1} \circ h \circ \sigma(X)=\rho^{-1} \circ \frac{X^Q(\bar{\ep}X+1)}{X+\ep} \circ \sigma(X)$, where $\rho(X):=\frac{\ep_1}{\bar{\ep}_1}\tilde{\rho}(X) \in \F_{q^2}(X)$ is of degree-one and permutes $\mu_{q+1}$, and $\ep=\frac{\ep_2}{\ep_1} \in \F_{q^2}^* \setminus \mu_{q+1}$. This proves 4) of Lemma \ref{lem6}. Now the proof of Lemma \ref{lem6} is complete.
\end{proof}

The proofs of Lemma \ref{lem5} and \ref{lem6} actually provided detailed conditions as to which family that $g(X)$ belongs. We record the results here for future references.
\begin{proposition} Let $\c \ne \underline{0}$. Assume that $g(X)$ is non-constant and $A(X)$ has no roots in $\mu_{q+1}$. Then both $U(X)$ and $V(X)$ as defined in Eq (\ref{3:WUV}) are nonzero. Moreover, denote by $\Gamma$ the set of roots of $V(X)$, and let $C(X)=\gcd(A(X),B(X))$.
\begin{itemize}
\item[(i)] If $U(X)/V(X) \in \F_{q^2}^*$ is constant, then
\begin{itemize}
\item[(1)] $\Gamma \subset \mu_{q+1}  \iff $ $g(X)$ belongs to Family 1) of Lemma \ref{lem5};

\item[(2)] $\Gamma \cap \mu_{q+1}=\emptyset , C(X)=1  \iff  g(X)$ belongs to Family 2) of Lemma \ref{lem5};

\item[(3)] $\Gamma \cap \mu_{q+1}=\emptyset, C(X) \ne 1  \iff  g(X)$ belongs to Family 3) of Lemma \ref{lem5};

\end{itemize}
\item[(ii)] If $U(X)/V(X)$ is non-constant, then
\begin{itemize}
\item[(1)] $\gcd(U(X),V(X)) \ne 1 \iff g(X)$ belongs to Family 1) of Lemma \ref{lem6};

\item[(2)] $\gcd(U(X),V(X))=1, U(X)$ has a unique root in $\mu_{q+1}  \iff  g(X)$ belongs to Family 2) of Lemma \ref{lem6};

\item[(3)] $\gcd(U(X),V(X))=1, U(X)$ has two distinct roots in $\mu_{q+1}  \iff  g(X)$ belongs to Family 3) of Lemma \ref{lem6};

\item[(4)] $\gcd(U(X),V(X))=1, U(X)$ has no roots in $\mu_{q+1}  \iff  g(X)$ belongs to Family 4) of Lemma \ref{lem6}.

\end{itemize}
\end{itemize}
\end{proposition}

\section{Linear Equivalence Classes of $f(X)$}\label{AEf}
In this section we turn Lemmas \ref{lem5} and \ref{lem6} about the classification of $g(X)$ into linear equivalence classes of $f(X)$. We shall remark that the idea is already implicit in the proof of \cite[Theorem 1.2]{Ding4}.


\begin{lemma}\label{lem7}
We use notation from Section \ref{Cla}. For $\c \ne \underline{0}$, suppose $\deg g=Q+1$.
    \begin{enumerate}
        \item If $g(X)=\rho^{-1} \circ \frac{A_{\c_1}(X)}{A_{\c_0}(X)} \circ \sigma$ for some degree-one $\rho, \sigma \in \F_{q^2}(X)$ both of which map $\mu_{q+1}$ to $\P^1(\F_q)$ and some $\c_0, \c_1 \in \F_q^4$, then $f(X)$ is linear equivalent to the $(Q,Q)$-biprojective function $P: \F_q^2 \to \F_q^2$ defined by $$P(x,y)=\left(y^{Q+1}A_{\c_1}(x/y),y^{Q+1}A_{\c_0}(x/y)\right).$$
        \item If $g=\rho^{-1} \circ g_{\c_0}(X) \circ \sigma$ for some degree-one $\rho, \sigma \in \F_{q^2}(X)$ both of which permute $\mu_{q+1}$ and some $\c_0 \in \F_{q^2}^4$, then $f(X)$ is linear equivalent to $f_{\c_0}(X)$.
    \end{enumerate}
\end{lemma}
\begin{proof}
1). Since $\rho, \sigma$ map $\mu_{q+1}$ to $\P^1(\F_q)$, there exist $\alpha,\beta \in \F_{q^2} \setminus \F_q$ and $\gamma,\delta \in \mu_{q+1}$ such that
        $$\sigma(X)=\frac{\alpha X+\gamma\bar{\alpha}}{X+\gamma}$$
        and
        $$\rho(X)=\frac{\beta X+\delta\bar{\beta}}{X+\delta}.$$
        Clearly
                $$\rho^{-1}(X)=-\frac{\delta(X-\bar{\beta})}{X-\beta}.$$
        We may expand the expression $g(X)=\rho^{-1} \circ \frac{A_{\c_1}(X)}{A_{\c_0}(X)} \circ \sigma$ and obtain
        \begin{align*}
            g(X)&=-\frac{\delta(X-\bar{\beta})}{X-\beta} \circ \frac{A_{\c_1}(X)}{A_{\c_0}(X)} \circ \frac{\alpha X+\gamma\bar{\alpha}}{X+\gamma}\\
            &=-\frac{\delta\left[A_{\c_1}\left(\frac{\alpha X+\gamma\bar{\alpha}}{X+\gamma}\right)-\bar{\beta}A_{\c_0}\left(\frac{\alpha X+\gamma\bar{\alpha}}{X+\gamma}\right)\right]}{A_{\c_1}\left(\frac{\alpha X+\gamma\bar{\alpha}}{X+\gamma}\right)-\beta A_{\c_0}\left(\frac{\alpha X+\gamma\bar{\alpha}}{X+\gamma}\right)}.
        \end{align*}
        Multiplying $(X+\gamma)^{Q+1}$ on both the numerator and denominator of the right side, we obtain
        \begin{align} \label{5:gx}
            g(X)&=-\frac{\delta\gamma^{Q+1}X^{Q+1}D^{(q)}(1/X)}{D(X)},
        \end{align}
where $$D(X):=(X+\gamma)^{Q+1}\left[A_{\c_1}\left(\frac{\alpha X+\gamma\bar{\alpha}}{X+\gamma}\right)-\beta A_{\c_0}\left(\frac{\alpha X+\gamma\bar{\alpha}}{X+\gamma}\right)\right].$$ Since $A_{\c_1}(X),A_{\c_0}(X) \in \F_q[X]$, we have $D(X)  \in \F_{q^2}[X]$.

Comparing (\ref{5:gx}) with the expression $g(X)=\frac{B(X)}{A(X)}$ and using the condition $\deg g=Q+1$, we shall have
\begin{eqnarray*} \max\left\{\deg B,\deg A\right\}&=&\max\left\{\deg X^{Q+1}D^{(q)}(1/X), \deg D\right\}=Q+1,\\ \gcd\left(B(X),A(X)\right)&=&\gcd\left(X^{Q+1}D^{(q)}(1/X),D(X)\right)=1.
 \end{eqnarray*}
This implies that $$A(X)=\lambda D(X),$$
for some $\lambda \in \F_{q^2}^*$. Now using
$$f(X)=X^{Q+1}A\left(X^{q-1}\right)=X^{Q+1}A\left(\bar{X}/X\right),$$
we obtain
$$f(X)=\lambda (\bar{X}+\gamma X)^{Q+1}\left[A_{\c_1}\left(\frac{\alpha \bar{X}+\gamma\bar{\alpha}X}{\bar{X}+\gamma X}\right)-\beta A_{\c_0}\left(\frac{\alpha\bar{X}+\gamma\bar{\alpha} X}{\bar{X}+\gamma X}\right)\right].$$

The above expression of $f(X)$ can be further simplified. Writing $\gamma=\bar{\ep}/\ep$ for some $\ep \in \F_{q^2}^*$, and letting
\begin{eqnarray} \left\{\begin{array}{lll} x&=&\ep\alpha\bar{X}+\bar{\ep}\bar{\alpha}X, \\ y&=&\ep\bar{X}+\bar{\ep}X,
\end{array} \right.
\end{eqnarray}
clearly $x,y \in \F_q$ for any $X \in \F_{q^2}$, we can write $f(X)$ as
        $$f(X)=\frac{\lambda}{\ep^{Q+1}}y^{Q+1}\left[A_{\c_1}\left(\frac{x}{y}\right)-\beta A_{\c_0}\left(\frac{x}{y}\right)\right]=L_1 \circ P \circ L_2(X),$$
        where
        $$L_1(x,y)=\frac{\lambda}{\ep^{Q+1}}\left(x-\beta y\right): \F_q^2 \to \F_{q^2}$$
        and
        $$L_2(X)=(x,y)=\left(\ep\alpha\bar{X}+\bar{\ep}\bar{\alpha}X,\ep\bar{X}+\bar{\ep}X\right): \F_{q^2} \to \F_q^2.$$
        Clearly $L_1$ and $L_2$ are linear functions. They are also permutation: $L_1$ is bijective since $\lambda \ep \ne 0$ and $\beta \in \F_{q^2} \setminus \F_q$; $L_2$ is bijective since
        $$\det \begin{bmatrix}
            \ep\alpha & \bar{\ep}\bar{\alpha}\\
            \ep & \bar{\ep}
        \end{bmatrix} =\ep\bar{\ep}(\alpha-\bar{\alpha}) \ne 0.$$
        Therefore $f$ is linear equivalent to $P$ over $\F_{q^2}$.

2). Since $\rho,\sigma$ permute $\mu_{q+1}$, there exist $\alpha_i,\beta_i \in \F_{q^2}$ for $i=1,2$ such that $\alpha_1\bar{\alpha}_1 \neq \alpha_2\bar{\alpha}_2, \beta_1\bar{\beta}_1 \neq \beta_2\bar{\beta}_2$,
        $$\sigma(X)=\frac{\bar{\alpha}_2X+\bar{\alpha}_1}{\alpha_1X+\alpha_2}, \quad \rho(X)=\frac{\bar{\beta}_2X+\bar{\beta}_1}{\beta_1X+\beta_2}.$$
        As $g_{\c_0}(X)=\frac{B_{\c_0}(X)}{A_{\c_0}(X)}$ and $g(X)=\rho^{-1} \circ g_{\c_0}(X) \circ \sigma$, we have
        \begin{align*}
            g(X)&=-\frac{\beta_2X-\bar{\beta}_1}{\beta_1X-\bar{\beta}_2} \circ \frac{B_{\c_0}(X)}{A_{\c_0}(X)} \circ \frac{\bar{\alpha}_2X+\bar{\alpha}_1}{\alpha_1X+\alpha_2}\\
            &=-\frac{\beta_2B_{\c_0}\left(\frac{\bar{\alpha}_2X+\bar{\alpha}_1}{\alpha_1X+\alpha_2}\right)
            -\bar{\beta}_1A_{\c_0}\left(\frac{\bar{\alpha}_2X+\bar{\alpha}_1}{\alpha_1X+\alpha_2}\right)}{\beta_1B_{\c_0}\left(\frac{\bar{\alpha}_2X+\bar{\alpha}_1}{\alpha_1X+\alpha_2}\right)-\bar{\beta}_2A_{\c_0}\left(\frac{\bar{\alpha}_2X+\bar{\alpha}_1}{\alpha_1X+\alpha_2}\right)}.
            \end{align*}
        Similarly we can obtain
        \begin{align*}
        g(X)&=\frac{X^{Q+1}D^{(q)}(1/X)}{D(X)},
        \end{align*}
        where $$D(X):=(\alpha_1X+\alpha_2)^{Q+1}\left[\beta_1B_{\c_0}\left(\frac{\bar{\alpha}_2X+\bar{\alpha}_1}{\alpha_1X+\alpha_2}\right)-\bar{\beta}_2A_{\c_0}\left(\frac{\bar{\alpha}_2X+\bar{\alpha}_1}{\alpha_1X+\alpha_2}\right)\right]\in \F_{q^2}[X].$$ Also using similar argument as in 1), and noting that $\deg g=Q+1$, we have $A(X)=\gamma D(X)$ for some $\gamma \in \F_q^*$. This further implies that
            \begin{align*}
            f(X)&=\gamma(\alpha_1\bar{X}+\alpha_2X)^{Q+1}\left[\beta_1B_{\c_0}\left(\frac{\bar{\alpha}_2\bar{X}+\bar{\alpha}_1X}{\alpha_1\bar{X}+\alpha_2X}\right)-\bar{\beta}_2A_{\c_0}\left(\frac{\bar{\alpha}_2\bar{X}+\bar{\alpha}_1X}{\alpha_1\bar{X}+\alpha_2X}\right)\right]\\
            &=L_1 \circ f_{\c_0}(X) \circ L_2,
            \end{align*}
            where $L_1,L_2$ are given by
            $$L_1(X)=\gamma(\beta_1\bar{X}-\bar{\beta}_2X), \quad L_2(X)=\alpha_1\bar{X}+\alpha_2X.$$
        The maps $L_1,L_2: \F_{q^2} \to \F_{q^2}$ are linear permutations because $\alpha_1\bar{\alpha}_1 \neq \alpha_2\bar{\alpha}_2$ and $$(\gamma\beta_1)(\overline{\gamma\beta_1})=\gamma^2\beta_1\bar{\beta}_1\neq \gamma^2\beta_2\bar{\beta}_2=(-\gamma\bar{\beta}_2)(\overline{-\gamma\bar{\beta}_2}).$$
        Hence $f(X)$ is linear equivalent to $f_{\c_0}(X)$ over $\F_{q^2}$. This completes the proof of Lemma \ref{lem7}.
\end{proof}
Assuming that $g(X)$ is non-constant and $A(X)$ has no roots in $\mu_{q+1}$, we can list all the possible linear equivalence classes of $f(X)$ as follows.
\begin{theorem}\label{AE2}
    If $g(X)$ is non-constant and $A(X)$ has no roots in $\mu_{q+1}$, then $f(X)$ is linear equivalent to one of the following functions:
    \begin{enumerate}
        \item $P_0(x,y)=(x^{Q+1},y^{Q+1}): \F_q^2 \to \F_q^2$;
        \item $f_0(X)=X^{Q+1}$;
        \item $f_1(X)=X^{Q+q}$;
        \item $P_1(x,y)=(x^{Q+1},xy^Q+y^{Q+1}): \F_q^2 \to \F_q^2$;
        \item $P_2(x,y)=(x^Q y, x^{Q+1}+\ep y^{Q+1}): \F_q^2 \to \F_q^2$ for some $\ep \in \F_q^*$;
        \item $P_3(x,y)=(x^{Q+1}-x^Q y, xy^Q+\ep y^{Q+1}): \F_q^2 \to \F_q^2$ for some $\ep \in \F_q^* \setminus \{-1\}$;
        \item $f_2(X)=X^{Q+q}+\ep X^{Q+1}$ for some $\ep \in \F_{q^2}^* \setminus \mu_{q+1}$.
    \end{enumerate}
\end{theorem}
\begin{proof}
By 3) of Lemma \ref{lem11}, we know that $g(X)$ must be in one of the families in Lemmas \ref{lem5} or \ref{lem6}.

We first look at the cases when $g(X)$ is in Family 1) of Lemma \ref{lem5} or Families 1) to 3) of Lemma \ref{lem6}. In all these four families $\deg g=Q+1$ and $g(X)$ is of the form $\rho^{-1} \circ \frac{A_{\c_1}(X)}{A_{\c_0}(X)} \circ \sigma$ for some $\c_0, \c_1 \in \F_q^4$ and degree-one $\rho, \sigma \in \F_{q^2}(X)$ both of which map $\mu_{q+1}$ to $\P^1(\F_q)$. Hence 1) of Lemma \ref{lem7} applies, and $f(X)$ is linear equivalent to functions listed in 1), 4), 5) and 6) of Theorem \ref{AE2} respectively.

    Now consider the case when $g(X)$ is in Family 2) of Lemma \ref{lem5} or Family 4) of Lemma \ref{lem6}. In both families $\deg g=Q+1$ and $g(X)$ is of the form $\rho^{-1} \circ g_{\c_0}(X) \circ \sigma$ for some degree-one $\rho, \sigma \in \F_{q^2}(X)$ both of which permute $\mu_{q+1}$, and $\c_0=(0,0,0,1)$ and $(0,0,1,\ep)$ respectively. Hence 2) of Lemma \ref{lem7} applies, and $f(X)$ is linear equivalent to functions listed in  2) and 7) of Theorem \ref{AE2} respectively.

    It remains to investigate the case when $g(X)$ is in Family 3) of Lemma \ref{lem5}. Here $\deg g=Q-1$, so we cannot apply Lemma \ref{lem7} directly. Instead, we can find the linear equivalence of $f(X)$ directly, as in the proofs of Lemma \ref{lem7}. Writing $\sigma, \rho$ explicitly as
        $$\sigma(X)=\frac{\bar{\alpha}_2X+\bar{\alpha}_1}{\alpha_1X+\alpha_2}, \quad \rho(X)=\frac{\bar{\beta}_2X+\bar{\beta}_1}{\beta_1X+\beta_2}, $$
        where $\alpha_i,\beta_i \in \F_{q^2}$ for $i=1,2$ such that $\alpha_1\bar{\alpha}_1 \neq \alpha_2\bar{\alpha}_2, \beta_1\bar{\beta}_1 \neq \beta_2\bar{\beta}_2$, we have
        \begin{align*}
            g(X)&=\rho^{-1} \circ X^{Q-1} \circ \sigma=-\frac{\beta_2X-\bar{\beta}_1}{\beta_1X-\bar{\beta}_2} \circ X^{Q-1} \circ \frac{\bar{\alpha}_2X+\bar{\alpha}_1}{\alpha_1X+\alpha_2}\\
            &=-\frac{\beta_2(\bar{\alpha}_2X+\bar{\alpha}_1)^{Q-1}-\bar{\beta}_1(\alpha_1X+\alpha_2)^{Q-1}}{\beta_1(\bar{\alpha}_2X+\bar{\alpha}_1)^{Q-1}-\bar{\beta}_2(\alpha_1X+\alpha_2)^{Q-1}}.
            \end{align*}
        This can be further written as
        \begin{align} \label{4:g2x}
            g(X) &=\frac{X^{Q+1}D^{(q)}(1/X)}{D(X)},
            \end{align}
           where $$D(X):=(\alpha_1X+\alpha_2)(\bar{\alpha}_2X+\bar{\alpha}_1)\left[\beta_1(\bar{\alpha}_2X+\bar{\alpha}_1)^{Q-1}-\bar{\beta}_2(\alpha_1X+\alpha_2)^{Q-1}\right].$$
            Again by using $g(X)=\frac{B(X)}{A(X)}$ and $\deg g=Q-1$, we conclude that $A(X)=\gamma D(X)$ for some $\gamma \in \F_q^*$.
            Now from $f(X)=X^{Q+1}A \left(\bar{X}/X\right)$, we have
            \begin{align*}
            f(X)&=\gamma(\alpha_1\bar{X}+\alpha_2X)(\bar{\alpha}_2\bar{X}+\bar{\alpha}_1X)\left[\beta_1(\bar{\alpha}_2\bar{X}+\bar{\alpha}_1X)^{Q-1}-\bar{\beta}_2(\alpha_1\bar{X}+\alpha_2X)^{Q-1}\right]\\
            &=\gamma[\beta_1(\alpha_1\bar{X}+\alpha_2X)(\bar{\alpha}_2\bar{X}+\bar{\alpha}_1X)^Q-\bar{\beta}_2(\alpha_1\bar{X}+\alpha_2X)^Q(\bar{\alpha}_2\bar{X}+\bar{\alpha}_1X)]\\
            &=L_1 \circ X^Q\bar{X} \circ L_2(X),
            \end{align*}
             where
            $$L_1(X)=\gamma\left(\beta_1\bar{X}-\bar{\beta}_2X\right)$$
            and
            $$L_2(X)=\alpha_1\bar{X}+\alpha_2X.$$
            Since $\alpha_1\bar{\alpha}_1 \neq \alpha_2\bar{\alpha}_2$ and $$(\gamma\beta_1)(\overline{\gamma\beta_1})=\gamma^2\beta_1\bar{\beta}_1\neq \gamma^2\beta_2\bar{\beta}_2=(-\gamma\bar{\beta}_2)(\overline{-\gamma\bar{\beta}_2}),$$
            the maps $L_1,L_2: \F_{q^2} \to \F_{q^2}$ are linear permutations, and hence $f(X)$ is linear equivalent to $X^Q\bar{X}$, that is, the function in 3) of Theorem \ref{AE2}. Now the proof of Theorem \ref{AE2} is complete.

\end{proof}
We remark here that all results in this section work for $p=2$ as well. This classification also provides an alternative proof to the main result of \cite{Faruk2} that all permutation quadrinomials $f_{\c}(X)$ given in (\ref{DO}) are linear equivalent to Gold or ``doubly-Gold'' functions.


\section{Proof of Theorem \ref{Main}}\label{ProMain}
Let us first define two-to-one functions.

\begin{definition}
    Given finite sets $R$ and $S$, a map $F:R \to S$ is said to be \textbf{two-to-one} (or 2-to-1 for short) if one of the following conditions holds:
\begin{enumerate}
    \item if $\sharp R$ is even, then for any $s \in S$, $\sharp F^{-1}(s) \in \{0,2\}$;
    \item if $\sharp R$ is odd, then there is a unique $s_0 \in S$ such that $\sharp F^{-1}(s_0)=1$, and for any $s \in S \setminus \{s_0\}$, $\sharp F^{-1}(s) \in \{0,2\}$.
\end{enumerate}
\end{definition}
The following is a result relating two-to-one and planarity for certain kind of polynomials. It is extracted from \cite[Theorem 1.1]{Chen}.
\begin{theorem} \label{6:DO21}  Let $\F_q$ be a finite field of odd order $q$ and $F(x)$ be a DO polynomial over $\F_q$. The following statements are equivalent:
\begin{itemize}
\item[(i)] $F(x)$ is planar;
\item[(ii)] $F(x)$ is a two-to-one map with $f^{-1}(0)=\{0\}$.
    \end{itemize}
    \end{theorem}

Since $f_{\c}(X)$ is a DO polynonmial, by using Theorem \ref{6:DO21}, we can prove the following.

\begin{lemma}\label{2211}
    If $f_{\c}(X)$ is planar, then $g(X)$ is non-constant and $A(X)$ has no roots in $\mu_{q+1}$.
\end{lemma}
\begin{proof}
Suppose $A(\alpha)=0$ for some $\alpha \in \mu_{q+1}$. Take an $x_0 \in \F_{q^2}^*$ such that $x_0^{q-1}=\alpha$. Then for any $t \in \F_q^*$ we have
\[f_{\c}(tx_0)=(tx_0)^{Q+1}A\left((tx_0)^{q-1}\right)=(tx_0)^{Q+1}A\left(\alpha\right)=0,\]
thus the equation
\[f_{\c}(x+x_0)-f_{\c}(x)=0\]
has solutions $x=tx_0$ for any $t \in \F_q^*$, that is, $f_{\c}(X)$ is not planar.

Next suppose $g(X)=\frac{B(X)}{A(X)}=\lambda$ is a constant. For any $x \in \F_{q^2}^*$, suppose $x^{q-1}=\alpha$, we have
\[f(x)^{q-1}=x^{(q-1)(Q+1)}A \left(x^{q-1}\right)^{q-1}=\alpha^{Q+1}A(\alpha)^{q-1}. \]
Since $\alpha \in \mu_{q+1}$, the right hand side can be written as
\[\alpha^{Q+1}A(\alpha)^{q-1}=\alpha^{Q+1} \frac{A(\alpha)^q}{A(\alpha)}=\frac{B(\alpha)}{A(\alpha)}=\lambda.\]
Since $x$ is arbitrary, we obtain
\[f(x)^{q-1}=\lambda, \qquad \forall x \in \F_{q^2}^*. \]
This implies that $f(x)$ takes at most $q-1$ distinct values as $x$ varies in the set $\F_{q^2}^*$ of cardinality $q^2-1$, so there is at least one value $y \in \F_{q^2}^*$ with $y^{q-1}=\lambda$ such that $\sharp f_{\c}^{-1}(y) \ge \frac{q^2-1}{q-1}=q+1 \ge 4$, so $f_{\c}(X)$ is not 2-to-1 over $\F_{q^2}^*$. Since $f_{\c}(X)$ is a DO polynomial, by Theorem \ref{6:DO21}, $f_{\c}(X)$ is not planar. This completes the proof of Lemma \ref{2211}.
\end{proof}

If $f_{\c}(X)$ is planar, Lemma \ref{2211} ensures that $g(X)$ is nonconstant and $A(X)$ has no roots in $\mu_{q+1}$, then by Theorem \ref{AE2}, we know that $f_{\c}(X)$ must be in one of the linear equivalence classes 1)--7). Note that 2-to-1 and planar properties are preserved under linear equivalence. Therefore it suffices to check the seven functions listed in Theorem \ref{AE2} to see whether or not they are planar or 2-to-1 (the number labels correspond to the family numbers in Theorem \ref{AE2}). In this section we focus on Cases 1)-5), which are relatively straightforward. We leave the more difficult Cases 6) and 7) to {\bf Appendix}.

\subsection{Cases 1)-5)}

{\bf Case 1).} $P_0(x,y)=\left(x^{Q+1},y^{Q+1}\right)$: since $Q+1$ is even, $P_0(\pm 1,\pm 1)=(1,1)$, that is, $\sharp P_0^{-1}((1,1)) \ge 4$, so $P_0$ is not 2-to-1 and hence not planar.

{\bf Case 2).} $f_0(X)=X^{Q+1}$: $f_0$ is 2-to-1 on $\F_{q^2}$ if and only if $\gcd(Q+1,q^2-1)=2$, which holds if and only if $\frac{2k}{\gcd(2k,\l)}$ is odd by Lemma \ref{2:gcd}, since $q=p^k,Q=p^{\l}$ and $p$ is odd. With a little further computation, this condition is then equivalent to $\frac{\l}{\gcd(k,\l)}$ being even.

{\bf Case 3).} $f_1(X)=X^{Q+q}$: $f_1$ is 2-to-1 one $\F_{q^2}$ if and only if $\gcd(Q+q,q^2-1)=2$, which holds if and only if $\gcd(Qq+1,q^2-1)=2$ if and only if $\frac{2k}{\gcd(k+\l,2k)}$ is odd by Lemma \ref{2:gcd}, since $q=p^k,Q=p^{\l}$ and $p$ is odd. With a little further computation, this condition is then equivalent to $\frac{k\l}{\gcd(k,\l)^2}$ being odd.

{\bf Case 4).} $P_1(x,y)=\left(x^{Q+1},xy^Q+y^{Q+1}\right)$: it is easy to check that $P(\pm 1,0)=P(1,-1)=P(-1,1)=(1,0)$, so $P_1$ is not 2-to-1 and hence not planar.

{\bf Case 5).} In this part we prove that $P_2(x,y)=(x^Qy,x^{Q+1}+\ep y^{Q+1})$ is planar on $\F_q^2$ if and only if
\begin{eqnarray} \label{57:cond}
\frac{k}{\gcd (k,\l)} \mbox{ is odd and } \ep \in \F_{q}^* \mbox{ is a non-square.}
\end{eqnarray}

Suppose $P_2(x,y)$ is planar or 2-to-1 on $\F_q^2$. Since $x^{Q+1}=1$ has $d=\gcd (Q+1,q-1) \ge 2$ solutions $x \in \F_q^*$, say $x_1, \cdots,x_d$, and obviously $P_2(x_i,0)=(0,1)$ for any $i$, so $d=2$, that is, $\frac{k}{\gcd (k,\l)}$ is odd. Moreover, if $\ep \in \F_{q}^*$ is a square, then the equation $\ep y^{Q+1}=1$ has two solutions $y_1,y_2 \in \F_q^*$, and we have $P_2(\pm 1,0)=P_2(0,y_1)=P_2(0,y_2)=(0,1)$, so $P_2(x,y)$ cannot be 2-to-1 on $\F_q^2$. So if $P_2(x,y)$ is 2-to-1 on $\F_q^2$ then conditions (\ref{57:cond}) hold.

From now on let us assume conditions (\ref{57:cond}) hold. To prove that $P_2(x,y)$ is planar, we need to show that that, for any $(a,b) \in \F_q^2 \setminus \{(0,0)\}$ and any $(c,d) \in \F_q^2$, the equations $$P_2(x+a,y+b)-P_2(x,y)=(c,d),$$ or equivalently,
\begin{eqnarray} \label{57:eq1} \left\{\begin{array}{lll}
(x+a)^{Q}(y+b)-x^Qy&=&c,\\
(x+a)^{Q+1}+\ep(y+b)^{Q+1}-x^{Q+1}-\ep y^{Q+1}&=&d,
\end{array}
\right.\end{eqnarray}
always have a unique solution $(x,y) \in \F_q^2$ under conditions (\ref{57:cond}).

Eq (\ref{57:eq1}) can be simplified as
\begin{eqnarray*} \left\{\begin{array}{lll}
x^Qb+y a^Q&=&c-a^Qb,\\
x^Qa+x a^Q+\ep(y^Qb+yb^Q)&=&d-a^{Q+1}-\ep b^{Q+1}.
\end{array}
\right.\end{eqnarray*}
Since the left hand sides are $\F_p$-linear in variables $x$ and $y$, it suffices to prove that under conditions (\ref{57:cond}), for any $(a,b) \in \F_q^2 \setminus \{(0,0)\}$, the equations below always have a unique solution $(x,y) \in \F_q^2$ which is $(x,y)=(0,0)$:
\begin{eqnarray} \label{57:eq3} \left\{\begin{array}{lll}
x^Qb+y a^Q&=&0,\\
x^Qa+x a^Q+\ep(y^Qb+yb^Q)&=&0.
\end{array}
\right.\end{eqnarray}
{\bf Case i} If $a \ne 0, b=0$, it is easy to see that the first equation of (\ref{57:eq3}) implies that $y=0$. As for the second equation of (\ref{57:eq3}), making a change of variable $x \mapsto ax$, we have $a^{Q+1}(x^Q+x)=0$. Obviously $x=0$ is a solution. If $x \ne 0$, then $x^{Q-1}=-1$. However, denote $\delta=\gcd(\l,k)$, since $k/\delta$ is odd by conditions (\ref{57:cond}), we have
\begin{eqnarray*}
(-1)^{\frac{q-1}{\gcd(Q-1,q-1)}}=(-1)^{\frac{p^k-1}{\gcd(p^{\l}-1,p^k-1)}}=(-1)^{\frac{p^k-1}{p^{\delta}-1}}=-1,
\end{eqnarray*}
that is, $x^{Q-1}=-1$ is not solvable in $\F_q$. In the second equality above we apply Lemma \ref{2:gcd}. This also shows that $x^Q+x$ is a permutation on $\F_q$, that is, Eq (\ref{57:eq3}) have a unique solution in this case under conditions (\ref{57:cond}).

{\bf Case ii} If $a=0, b \ne 0$, this case is very similar to {\bf Case i}, showing Eq (\ref{57:eq3}) have a unique solution. We omit the details.

{\bf Case iii} If $ab \ne 0$, making a change of variables $x \mapsto ax, y \mapsto by$, we can simplify Eq (\ref{57:eq3}) as
\begin{eqnarray} \label{57:eq4} \left\{\begin{array}{lll}
a^Qb(x^Q+y)&=&0,\\
a^{Q+1}(x^Q+x)+\ep b^{Q+1}(y^Q+y)&=&0.
\end{array}
\right.\end{eqnarray}
The first equation of Eq (\ref{57:eq4}) implies that $y=-x^Q$. Plugging this into the second equation of (\ref{57:eq4}), and taking $z=x^Q+x$, we obtain
\begin{eqnarray} \label{57:eq5} \left\{\begin{array}{lll}
x^Q+x&=&z,\\
a^{Q+1}z-\ep b^{Q+1}z^Q&=&0.
\end{array}
\right.\end{eqnarray}
Since $k/\delta$ is odd, $x^Q+x$ is a permutation on $\F_q$, so the number of solutions $(x,y) \in \F_q^2$ of Eq (\ref{57:eq4}) is the same as the number of solutions $z \in \F_q$ satisfying the second equation of (\ref{57:eq5}), which we can write as
\begin{eqnarray*} \label{57:eq6}
z^Q=\left(\frac{a}{b}\right)^{Q+1} \frac{z}{\ep}.
\end{eqnarray*}
$z=0$ is obviously a solution. If $z \ne 0$, then
\begin{eqnarray} \label{57:eq7} z^{Q-1}=\left(\frac{a}{b}\right)^{Q+1} \frac{1}{\ep}.
\end{eqnarray}
Eq (\ref{57:eq7}) is not solvable for $z \in \F_q^*$ as the left hand side is a square in $\F_q^*$, but the right hand side is a nonsquare as $\ep \in \F_q^*$ is a nonsquare. This implies that the original Eq (\ref{57:eq3}) has only the unique solution $(x,y)=(0,0)$ in $\F_q^2$. This proves that $P_2(x,y)$ is planar under conditions (\ref{57:cond}).

{\bf Cases 6) and 7).} These cases are more technical. We will prove in {\bf Appendix} that functions from these two cases are never planar when $k \nmid \l$. This concludes the first part of Theorem \ref{Main}.

\subsection{Case $k \mid \l$}

Now we assume $k \mid \l$. This implies that $Q=p^\l=q^{\l/k}$. Then $X^Q$ induces the same function on $\F_{q^2}$ as $\bar{X}$ or $X$ according to whether $\frac{\l}{k}$ is odd or even. Therefore $f_{\c}(X)$ induces the same function on $\F_{q^2}$ as
\begin{equation}\label{tf}
F_{\c}(X):=\begin{cases}
    c_2\bar{X}^2+(c_0+c_3)X\bar{X}+c_1X^2 &(\frac{\l}{k} \text{ is odd})\\
    c_0\bar{X}^2+(c_1+c_2)X\bar{X}+c_3X^2 &(\frac{\l}{k} \text{ is even}).
\end{cases}
\end{equation}
This leads us to study a simpler form of polynomials, namely $$\tilde{f}(X)=\tilde{f}_{\a}(X)=a_0\bar{X}^2+a_1X\bar{X}+a_2X^2 \in \F_{q^2}[X]$$
for $\a=(a_0,a_1,a_2) \in \F_{q^2}^3$. We see that $\tilde{f}(X)=X^2\tilde{A}(X^{q-1})$ where $$\tilde{A}(X)=\tilde{A}_{\a}(X)=a_0X^2+a_1X+a_2.$$
Let us introduce some notations. For $\a=(a_0,a_1,a_2) \in \F_{q^2}^3$, define
\begin{align}
    \tilde{B}(X)&:=\bar{a}_2X^2+\bar{a}_1X+\bar{a}_0=X^2\tilde{A}^{(q)}(1/X), \nonumber \\
    \tilde{G}(X)&:=\frac{\tilde{B}(X)}{\tilde{A}(X)},\nonumber \\
    \tilde{e}&:=a_2\bar{a}_2-a_0\bar{a}_0,\label{7:e} \\
    \tilde{\te}&:=\bar{a}_1a_2-\bar{a}_0a_1. \label{7:te}
\end{align}
Then we have the following result, and the second part of Theorem \ref{Main} follows immediately.
\begin{lemma}\label{2210}
Let $\a=(a_0,a_1,a_2) \in \F_{q^2}^3$. Then $\tilde{f}(X)=\tilde{f}_{\a}(X)=a_0\bar{X}^2+a_1X\bar{X}+a_2X^2 \in \F_{q^2}[X]$ is 2-to-1 over $\F_{q^2}$ if and only if $\tilde{e}^2-\tilde{\te}^{q+1} \in \F_q^*$ is a square. Moreover, in this case $\tilde{f}(X)$ is linear equivalent to $X^2$.
\end{lemma}
\begin{proof}[Proof of Lemma \ref{2210}]
Since $\gcd(2,q-1)=2$, as similar to Lemma \ref{2211}, we see that $\tilde{A}(X)$ has no roots in $\mu_{q+1}$ and $\tilde{G}(X)$ is non-constant on $\mu_{q+1}$. Since $\deg\tilde{G} \leq 2$, we actually must have $\deg \tilde{G}=2$, that is, $\gcd(\tilde{A}(X),\tilde{B}(X))=1$. In particular, no points have ramification index divisible by $p=\mathrm{char}(\overline{\F}_q)$. By Hurwitz genus formula (Lemma \ref{RH}), $\tilde{G}$ must have exactly two branch points in $\P^1(\overline{\F}_q)$, each having a unique preimage in $\P^1(\overline{\F}_q)$. Hence in particular $\tilde{G}$ is $\overline{\F}_q$-linear equivalent to $X^2$.

Define $\tilde{V}(X):=\tilde{B}'(X)\tilde{A}(X)-\tilde{B}(X)\tilde{A}'(X)$. It is easy to see that
\[ \tilde{V}(X)=\tilde{\te}^qX^2+2\tilde{e}X+\tilde{\te}, \]
where $\tilde{\te}$ and $\tilde{e}$ are given in (\ref{7:te}) and (\ref{7:e}) respectively.
The ramification points of $\tilde{G}$ are the roots of $\tilde{V}$ in $\overline{\F}_q$ (and $\infty$ if $\deg \tilde{V}=1$).  Since $\tilde{G}$ is separable, we have $\tilde{V}(X) \neq 0$, that is, at least one of $\tilde{e}$ and $\tilde{\te}$ is nonzero. Moreover, $\tilde{G}$ has two distinct ramification points. Hence $\tilde{\Delta}:=\frac{1}{4}\Delta(\tilde{V})=\tilde{e}^2-\tilde{\te}^{q+1} \neq 0$. Furthermore, by Lemma \ref{lem1}, $\tilde{V}$ has no roots in $\mu_{q+1}$ if and only if $\tilde{\Delta}$ is a square in $\F_q^*$.

We first assume $\tilde{\Delta}$ is a non-square in $\F_q$, then both ramification points of $\tilde{G}$ are in $\mu_{q+1}$. Following the proof of Lemma \ref{lem5} in Section \ref{Cla}, we see that $\tilde{G}(X)=\rho^{-1} \circ X^2 \circ \sigma$ for some degree-one $\rho,\sigma \in \F_{q^2}[X]$ mapping $\mu_{q+1}$ onto $\P^1(\F_q)$. Following the proof of Lemma \ref{lem7} in Section \ref{AEf}, we see that $\tilde{f}(X)$ is linear equivalent to the function $(x,y) \mapsto (x^2,y^2): \F_q^2 \to \F_q^2$. This map is not 2-to-1 since the four distinct points $(1,1),(-1,1),(1,-1),(-1,-1) \in \F_q^2$ are all mapped to the same point $(1,1) \in \F_q^2$.

Now we assume $\tilde{\Delta}$ is a square in $\F_q^*$, so neither ramification points of $\tilde{G}$ are in $\mu_{q+1}$. Following the proof of Lemma \ref{lem5} in in Section \ref{Cla}, we see that $\tilde{G}(X)=\rho^{-1} \circ X^2 \circ \sigma$ for some degree-one $\rho,\sigma \in \F_{q^2}[X]$ permuting $\mu_{q+1}$. Following the proof of Lemma \ref{lem7} in Section \ref{AEf}, we see that $\tilde{f}(X)$ is linear equivalent to $X^2$, which is indeed 2-to-1 and hence planar. This completes the proof of Lemma \ref{2210}.
\end{proof}

\section{Appendix}\label{Appen}
In this Appendix we treat Cases 6) and 7) from Theorem \ref{AE2}. We assume that $k \nmid \l$, as the case $k \mid \l$ is covered by the second part of Theorem \ref{Main}.

\subsection{Case 6) is NOT planar}
In this subsection we prove that the function $P_3(x,y)=(x^{Q+1}-x^{Q}y,xy^{Q}+\ep y^{Q+1}): \F_q^2 \to \F_q^2$ is never planar for any $\ep \in \F_q^* \setminus \{-1\}$ if $k \nmid \l$.

For any $a,b,c,d \in \F_q$, denote by $N$ the number of $(x,y) \in \F_q^2$ satisfying $$P_3(x+a,y+b)-P_3(x,y)=(c,d),$$ that is,
\begin{eqnarray} \label{7:eq1} \left\{\begin{array}{lll}
(x+a)^{Q+1}-(x+a)^Q(y+b)-\left(x^{Q+1}-x^Qy\right)&=&c,\\
(x+a)(y+b)^{Q}+\ep (y+b)^{Q+1}-\left(xy^{Q}+\ep y^{Q+1}\right)&=&d.
\end{array}
\right.\end{eqnarray}
$P_3(x,y)$ is planar if and only if $N=1$ for any $(a,b) \in \F_q^2 \setminus \{(0,0)\}$ and any $(c,d) \in \F_q^2$.

Eq (\ref{7:eq1}) can be simplified as
\begin{eqnarray} \label{7:eq2} \left\{\begin{array}{lll}
x^Qa+xa^Q+a^{Q+1}-(x^Qb+a^Qy+a^Qb)&=&c,\\
xb^Q+ay^Q+ab^Q+\ep(y^Qb+yb^Q+b^{Q+1})&=&d.
\end{array}
\right.\end{eqnarray}

Now suppose $ab \ne 0$. Taking
\begin{eqnarray*} \label{7:cd} c=a^{Q+1}-a^Qb, \quad d=ab^Q+\ep b^{Q+1}, \end{eqnarray*}
and making a change of variables $x \mapsto ax, y \mapsto by$, we can simplify Eq (\ref{7:eq2}) as
\begin{eqnarray} \label{7:eq3} \left\{\begin{array}{lll}
x^Q\left(1-\frac{b}{a}\right)+x-\frac{b}{a}y&=&0,\\
y^Q\left(1+\frac{\ep b}{a}\right)+\frac{\ep b}{a}y+x&=&0.
\end{array}
\right.\end{eqnarray}

We assume further that
\begin{eqnarray*} \label{7:ab}
ab\left(1-\frac{b}{a}\right)\left(1+\frac{\ep b}{a}\right) \ne 0.
\end{eqnarray*}
Denote
\begin{eqnarray*} \label{7:tau12}
t:=\frac{b}{a}, \quad \tau_1:=\frac{1}{1-t}, \quad \tau_2:=\frac{\ep t}{1+\ep t}.
\end{eqnarray*}
Then $t$ satisfies $t \in \F_q \setminus \{0,1,-\ep^{-1}\}$ and Eq (\ref{7:eq3}) is equivalent to the equations
\begin{eqnarray} \label{7:eq33} \left\{\begin{array}{lll}
x^Q+\tau_1 x+(1-\tau_1)y&=&0,\\
y^Q+\tau_2y+(1-\tau_2)x&=&0.
\end{array}
\right.\end{eqnarray}

Again we make a change of variables $x+y=X, x-y=Z$. Since $q$ is odd, this transformation $(x,y) \mapsto (X,Z)$ from $\F_q^2$ to itself is invertible. Adding and subtracting the two equations in (\ref{7:eq33}) and using the new variables $X,Z$ we obtain
\begin{eqnarray} \label{7:eq4} \left\{\begin{array}{lll}
X^Q+X+(\tau_1-\tau_2)Z&=&0,\\
Z^Q+(\tau_1+\tau_2-1)Z&=&0.
\end{array}
\right.\end{eqnarray}
In summary, if $P_3(x,y)$ is planar, then for any $t \in \F_q\setminus \{0,1,-\ep^{-1}\}$, there is always a unique solution $(X,Z) \in \F_q^2$ satisfying Eq (\ref{7:eq4}). We will prove that this is impossible.

Now suppose $P_3(x,y)$ is planar. Note that Eq (\ref{7:eq4}) always has the solution $X=Z=0$ for any such $t$. Taking $Z=0$, this means that $X^Q+X=0$ has only the solution $X=0$, that is, $X^{Q-1}=-1$ is not solvable for $X \in \F_{q}^*$. Denote $\delta=\gcd(\l,k)$. By Lemma \ref{2:gcd}, this requires that
\begin{eqnarray*}
(-1)^{\frac{q-1}{\gcd(Q-1,q-1)}}=(-1)^{\frac{p^k-1}{\gcd(p^{\l}-1,p^k-1)}}=(-1)^{\frac{p^k-1}{p^{\delta}-1}}=-1,
\end{eqnarray*}
that is, $\frac{k}{\delta}$ is odd. Indeed when $\frac{k}{\delta}$ is odd, then $X^Q+X$ is a permutation on $\F_q$, so the first equation in (\ref{7:eq4}) always has a unique solution $X \in \F_q$ for any given $Z$ satisfying the second equation of (\ref{7:eq4}). Since $P_3(x,y)$ is planar, the second equation of (\ref{7:eq4}) must always have the unique solution $Z=0$ for any $t \in \F_q \setminus \{0,1,-\ep^{-1}\}$, that is, the equation $Z^{Q-1}=1-\tau_1-\tau_2$ is not solvable for $Z \in \F_q^*$. Since $\gcd(Q-1,q-1)=p^{\delta}-1$, this is equivalent to saying that the equation
\begin{eqnarray*}
Z^{p^{\delta}-1}=1-\tau_1-\tau_2=\frac{-(\ep+1)t}{(1+\ep t)(1-t)}
\end{eqnarray*}
is not solvable for $Z \in \F_q$ for any $t$ satisfying $t \in \F_q\setminus \{0,1,-\ep^{-1}\}$. Denote by $\chi$ a multiplicative character of order $p^\delta-1$ defined on $\F_q^*$. Then we have
\[\frac{1}{p^{\delta}-1} \sum_{i=0}^{p^\delta-2} \chi^i \left(\frac{-(\ep+1)t}{(1+\ep t)(1-t)}\right)=0, \quad \forall t \in \F_q \setminus \{0,1,-\ep^{-1}\}. \]
So
\begin{eqnarray} \label{7:eqA} A:=\sum_{t \in \F_q \setminus \{0,1,-\ep^{-1}\}} \sum_{i=0}^{p^\delta-2} \chi^i \left(\frac{-(\ep+1)t}{(1+\ep t)(1-t)}\right)=0. \end{eqnarray}
On the other hand, isolating the term for $i=0$, we see that
\[A=q-3+\sum_{i=1}^{p^\delta-2} \sum_{t \in \P^1(\F_q)} \chi^i \left(\frac{-(\ep+1)t}{(1+\ep t)(1-t)}\right).\]
Here we use the convention that $\chi^i(0)=\chi^i(\infty)=0$ for the non-principal character $\chi^i$ $(1 \le i \le p^\delta-2)$. Since $\ep \in \F_q^* \setminus \{-1\}$, by the Weil bound (see Lemma \ref{2:weilbound})
\[\left|\sum_{t \in \P^1(\F_q)} \chi^i \left(\frac{-(\ep+1)t}{(1+\ep t)(1-t)}\right)\right| \le 2 \sqrt{q},\]
we find that
\[A\ge q-3-(p^\delta-2) 2 \sqrt{q}=p^k-2p^{\delta+k/2}+4p^{k/2}-3=:f(p,k,\delta).\]
Since $k \nmid \l$, we have $k/\delta \ge 3$ is odd, if $k \ge 6$, then $\delta \le \frac{k}{3}$, and we have
\[f(p,k,\delta) \ge p^{\frac{5k}{6}} \left(p^{\frac{k}{6}}-2\right)+4 p^{\frac{k}{2}}-3 >0, \quad \forall p \ge 3. \]
If $1<k <6$, then we must have $k=3$ or $5$ and $\delta=1$, we can also check easily that in these two cases $f(p,k,1)>0$ for any $p \ge 3$. From this we conclude that $A>0$, which contradicts (\ref{7:eqA}), implying that $P_3(x,y)$ cannot be planar on $\F_q^2$ for any $\ep \in \F_q^* \setminus \{-1\}$.

 \subsection{Case 7)}

 In this subsection we prove that the function $f_2(X)=X^{Q+q}+\ep X^{Q+1}$ is not planar on $\F_{q^2}$ for any $\ep \in \F_{q^2}^* \setminus \mu_{q+1}$ if $k \nmid \l$.

Suppose $f_2(X)$ is planar. For any $a,b \in \F_{q^2}$, denote by $N$ the number of $x \in \F_{q}^2$ satisfying $$f_2(x+a)-f_2(x)=b,$$ that is,
\begin{eqnarray} \label{8:eq1}
(x+a)^{Q+q}+\ep(x+a)^{Q+1}-\left(x^{Q+q}+\ep x^{Q+1}\right)&=&b.
\end{eqnarray}
Then $N=1$ for any $a \in \F_{q^2}^*$ and any $b \in \F_{q^2}$.

Eq (\ref{8:eq1}) can be simplified as
\begin{eqnarray} \label{8:eq2}
x^Qa^q+x^qa^Q+\ep(x^Qa+xa^Q)+a^{Q+q}+a^{Q+1}&=&b.
\end{eqnarray}
Now take $b= a^{Q+q}+a^{Q+1}$ and make a change of variables $x \mapsto ax$, we can simplify Eq (\ref{8:eq2}) as
\begin{eqnarray} \label{8:eq3}
x^Q\left(1+\ep a^{1-q}\right)+x^q+\ep a^{1-q}&=&0.
\end{eqnarray}
Since $\ep \in \F_{q^2}^* \setminus \mu_{q+1}$ and $a^{1-q} \in \mu_{q+1}$, we have $1+\ep a^{1-q} \ne 0$. Denote
\[\tau:=\frac{1}{1+\ep a^{1-q}}, \]
then Eq (\ref{8:eq3}) can be written as
\begin{eqnarray} \label{8:eq33}
x^Q+\tau x^q+(1-\tau)x&=&0.
\end{eqnarray}
Since $x \in \F_{q^2}$, taking a $q$-th power on both sides of Eq (\ref{8:eq33}) we have
\begin{eqnarray} \label{8:eq4}
\bar{x}^{Q}+(1-\bar{\tau}) \bar{x}+\bar{\tau}x&=&0.
\end{eqnarray}
Here we use $\bar{x}:=x^q$ for any $x \in \F_{q^2}$. Again we make a change of variables $y=x+\bar{x}, z=x-\bar{x}$. Since $q$ is odd and $x=\frac{y+z}{2}$, this transformation $x \mapsto (y,z)$ is invertible, where $y,z$ satisfy the conditions
\begin{eqnarray} \label{8:yz} y \in \F_q, \qquad \bar{z}=-z, \quad z \in \F_{q^2}.
\end{eqnarray}
Adding and subtracting both sides of Eq (\ref{8:eq33}) and Eq (\ref{8:eq4}), and using the new variables $y,z$ we obtain
\begin{eqnarray} \label{8:eq5} \left\{\begin{array}{lll}
y^Q+y+(-\tau+\bar{\tau})z&=&0,\\
z^Q+(1-\tau-\bar{\tau})z&=&0,
\end{array}
\right.\end{eqnarray}
and $y,z$ are subject to conditions (\ref{8:yz}).

In summary, since $f_2(X)$ is planar, for any $a \in \F_{q^2}^*$, the only solution $(y,z)$ of Eq (\ref{8:eq5}) subject to conditions (\ref{8:yz}) is $y=z=0$. We will prove that this is impossible.

Taking $z=0$ in the first equation of (\ref{8:eq5}), very similar to the proof of {\bf Case 6)}, we see that that $\frac{k}{\delta}$ must be odd where $q=p^k,Q=p^{\l}, \delta=\gcd(k,\l)$.

If $\frac{\l}{\delta}$ is even, taking a $q$-th power on $f_2(X)$ which we assume to be planar, we see that $\tilde{f}_2(X)=X^{\tilde{Q}+q}+\ep^{-q} X^{\tilde{Q}+1}$ is also planar where $\tilde{Q}=qQ=p^{\tilde{\l}}$ satisfying $\tilde{\l}=k+\l$. Then for this new planar function $\tilde{f}_2(X)$ we have $\delta=\gcd(k,\l)=\gcd(k,\tilde{\l})$, and both $\frac{k}{\delta}$ and $\frac{\tilde{\l}}{\delta}$ are odd. So by considering $\tilde{f}_2(X)$ if necessary, we may assume that $\frac{k\l}{\delta^2}$ is odd. By Lemma \ref{2:gcd}, this implies that
\begin{eqnarray} \label{8:qQ} p^\delta-1=\gcd(Q-1,q-1)=\gcd(Q-1,q^2-1).\end{eqnarray}

By considering $z \ne 0$ from Eq (\ref{8:yz}) and from the second equation of (\ref{8:eq5}), we conclude that since $f_2(X)$ is planar, the equations
\begin{eqnarray} \label{8:eq6}
z^{q-1}&=&-1,\\
\label{8:eq7}
z^{Q-1}&=&\alpha,
\end{eqnarray}
have no common solutions $z \in \F_{q^2}$ for any $a \in \F_{q^2}^*$. Here we denote
\begin{eqnarray*} \alpha=\tau+\bar{\tau}-1=\frac{1-\ep \bar{\ep}}{(1+\ep t)\left(1+\bar{\ep} t^{-1}\right)} \in \F_q^*, \quad t:=a^{1-q} \in \mu_{q+1}. \end{eqnarray*}
It is easy to see that as $a$ runs over $\F_{q^2}^*$, the value $t$ runs over the set $\mu_{q+1}$.

If $\alpha^{\frac{p^k-1}{p^{\delta}-1}}=-1$, since $q,Q$ satisfy conditions (\ref{8:qQ}), we have, by Lemma \ref{2:gcd},
\begin{eqnarray*} \alpha^{\frac{q^2-1}{\gcd(q^2-1,Q-1)}}&=&\alpha^{\frac{p^{2k}-1}{p^{\delta}-1}}=\left(\alpha^{\frac{p^k-1}{p^{\delta}-1}}\right)^{p^k+1}=(-1)^{p^k+1}=1,\\
\alpha^{\frac{q-1}{\gcd(q-1,Q-1)}}&=&\alpha^{\frac{p^{k}-1}{p^{\delta}-1}}=-1.\end{eqnarray*}
This implies that Eq (\ref{8:eq7}) is solvable for $z \in \F_{q^2} \setminus \F_q$. Let $z_0 \in \F_{q^2} \setminus \F_{q}$ be such a solution, then
\begin{eqnarray} \label{8:acon} -1=\alpha^{\frac{p^k-1}{p^{\delta}-1}}=\left(z_0^{Q-1}\right)^{\frac{p^k-1}{p^{\delta}-1}}=\left(z_0^{q-1}\right)^{\frac{Q-1}{p^{\delta}-1}}.  \end{eqnarray}
So $\left(z_0^{q-1}\right)^{2 \frac{Q-1}{p^\delta-1}}=1$. Together with $\left(z_0^{q-1}\right)^{q+1}=1$ and
$\gcd\left(q+1,2 \frac{Q-1}{p^\delta-1}\right)=\gcd(q+1,Q-1)=2$, which follows by applying Lemma \ref{2:gcd} to the assumption that $\frac{\l}{\delta}$ is odd, we obtain
$$\left(z_0^{q-1}\right)^2=1 \quad \Longrightarrow \quad z_0^{q-1} \in \{ \pm 1\}. $$
In view of (\ref{8:acon}), $z_0^{q-1} \ne 1$, so we must have $z_0^{q-1}=-1$, that is, $z_0 \in \F_{q^2} \setminus \F_q$ is a common solution of Eq (\ref{8:eq6}) and Eq (\ref{8:eq7}). Since $f_2(X)$ is planar, this cannot happen, that is, we must have
\begin{eqnarray} \label{8:atcon}
\alpha^{\frac{p^k-1}{p^{\delta}-1}}= \left(\frac{1-\ep \bar{\ep}}{(1+\ep t)\left(1+\bar{\ep} t^{-1}\right)}\right)^{\frac{p^k-1}{p^{\delta}-1}} \ne -1 \quad \forall t \in \mu_{q+1}.
\end{eqnarray}

Fixing an element $\xi \in \F_{q^2} \setminus \F_q$ and using the bijection $\rho(x)=\frac{x+\xi}{x+\bar{\xi}}$ to identify $\P^1(\F_q)$ and $\mu_{q+1}$, we can rewrite $\alpha$ as
\[\alpha=\frac{1-\ep\bar{\ep}}{\left(1+\ep \frac{x+\xi}{x+\bar{\xi}}\right)\left(1+\bar{\ep} \frac{x+\bar{\xi}}{x+\xi} \right)}=-\lambda_{\ep} \frac{N(x)}{D(x)}, \quad x \in \P^1(\F_q),
\]
where
\begin{eqnarray*}
\lambda_{\ep}:=\frac{\ep \bar{\ep}-1}{(1+\ep)(1+\bar{\ep})},
\end{eqnarray*}
and
\begin{eqnarray} \label{8:NDx}
N(x)=(x+\xi)(x+\bar{\xi}), \quad D(x)=\left(x+\frac{\bar{\xi}+\ep \xi}{1+\ep}\right) \left(x+\frac{\xi+\bar{\ep} \bar{\xi}}{1+\bar{\ep}}\right).
\end{eqnarray}
Since $\xi \ne \bar{\xi}$ and $\ep \bar{\ep} \ne 1$, it is easy to check that the four elements $\xi, \bar{\xi}, \frac{\bar{\xi}+\ep \xi}{1+\ep}$ and $\frac{\xi+\bar{\ep} \bar{\xi}}{1+\bar{\ep}}$ are all in $\F_{q^2} \setminus \F_q$ and are all distinct. Now Eq (\ref{8:atcon}) implies that
\begin{eqnarray*}
(-\alpha)^{\frac{p^k-1}{p^{\delta}-1}}=\left(\lambda_{\ep} \frac{N(x)}{D(x)}\right)^{\frac{p^k-1}{p^\delta-1}} \ne 1 \quad \forall x \in \P^1(\F_q).
\end{eqnarray*}
This is equivalent to saying that the equation
\begin{eqnarray*}
Z^{p^{\delta}-1}=\lambda_{\ep} \frac{N(x)}{D(x)}
\end{eqnarray*}
is not solvable for $Z \in \F_q$ for any $x \in \P^1(\F_q)$. Denote by $\chi$ a multiplicative character of order $p^\delta-1$ on $\F_q$. This means that
\[\frac{1}{p^{\delta}-1} \sum_{i=0}^{p^\delta-2} \chi^i \left(\lambda_{\ep} \frac{N(x)}{D(x)}\right)=0, \quad \forall x \in \P^1(\F_q). \]
So we have
\begin{eqnarray} \label{8:eqA} B:=\sum_{x \in \P^1(\F_q)} \sum_{i=0}^{p^\delta-2} \chi^i \left(\lambda_{\ep}\frac{N(x)}{D(x)}\right)=0. \end{eqnarray}
On the other hand, isolating the term for $i=0$, we see that
\[B=q+1+\sum_{i=1}^{p^\delta-2} \sum_{x \in \P^1(\F_q)} \chi^i \left(\lambda_{\ep}\frac{N(x)}{D(x)}\right).\]
Since $N(x),D(x)$ are of the form given in (\ref{8:NDx}), by the Weil bound (Lemma \ref{2:weilbound})
\[\left|\sum_{x \in \P^1(\F_q)} \chi^i \left(\lambda_{\ep} \frac{N(x)}{D(x)}\right)\right| \le 3 \sqrt{q}, \quad \forall 1 \le i \le p^\delta-2,\]
we find that
\[B \ge q+1-(p^\delta-2) 3 \sqrt{q}=p^k-3p^{\delta+k/2}+6p^{k/2}+1=:f(p,k,\delta).\]
We have
\[f(p,k,\delta)=p^{\delta+k/2} \left(p^{-\delta+k/2}-3+6p^{-\delta}\right)+1. \]
Since $k \nmid \l$, we have $k/\delta \ge 3$ is odd, if $k \ge 6$, then $\delta \le \frac{k}{3}$, it is easy to see that
\[f(p,k,\delta) \ge p^{\frac{5k}{6}} \left(p^{\frac{k}{6}}-3\right)+6 p^{\frac{k}{2}}+1 >0, \quad \forall p \ge 3. \]
If $1<k <6$, then we must have $k=3$ or $5$ and $\delta=1$, and we have
\[f(p,k,1)=p^{1+k/2} \left(p^{k/2-1}-3+6p^{-1}\right)+1. \]
One can verify easily that
\[f(p,5,1)>0, \quad f(p,3,1)>0, \quad \forall p \ge 3. \]
From this we conclude that $B>0$, which contradicts (\ref{8:eqA}), implying that $f_2(X)$ is not planar on $\F_{q^2}$ for any $\ep \in \F_{q^2}^* \setminus \mu_{q+1}$.

\section{Conclusion}\label{Con}

In this paper, we characterize planar functions from a class of quadrinomials in terms of linear equivalence. This supplements the classification result of APN functions \cite{Faruk2} from this class of quadrinomials in characteristic 2. The main ingredient of the paper is the ``geometric method'' developed by Ding and Zieve \cite{Ding4} in odd characteristic to study this problem, and the linear equivalence result followed naturally from this method. It may be interesting to see if other results can be obtained in this way.



\end{document}